\documentclass[12pt]{amsart}
\usepackage{amsmath,amsfonts,amsthm,amscd, amssymb,mathrsfs}
\usepackage{enumitem}
\usepackage[sort,nocompress]{cite}
\usepackage{hyperref}
\usepackage{tikz, tikz-cd}
\usepackage{bm}
\newtheorem{theorem}{Theorem}[section]

\newtheorem{proposition}[theorem]{Proposition}
\newtheorem{lemma}[theorem]{Lemma}
\newtheorem{corollary}[theorem]{Corollary}
\theoremstyle{definition}
\newtheorem{definition}[theorem]{Definition}
\newtheorem{remark}[theorem]{Remark}

\newcommand{\id}{{\rm{id}}}

\newcommand{\CP}{\mathbb{P}}
\newcommand{\PP}{\mathbb{P}}
\newcommand{\CC}{\mathbb{C}}
\newcommand{\RR}{\mathbb{R}}
\newcommand{\ZZ}{\mathbb{Z}}

\newcommand{\rank}{{\rm{rank}}}

\newcommand{\cO}{\mathcal O}
\newcommand{\cH}{\mathcal H}
\newcommand{\cE}{\mathcal E}
\newcommand{\cT}{\mathcal T}
\newcommand{\cL}{\mathcal L}

\newcommand{\sHom}{\mathcal{H}\!om}

\newcommand{\im}{\operatorname{im}}
\newcommand{\coker}{\operatorname{coker}}

\newcommand{\tr}{\operatorname{tr}}
\newcommand{\Span}{\operatorname{span}}

\DeclareMathOperator{\Hom}{Hom}

\usepackage[letterpaper]{geometry}
\numberwithin{equation}{section}
\newcommand{\cB}{\mathcal B}
\newcommand{\cF}{\mathcal F}

\newcommand{\OO}{\mathcal O}

\newcommand{\Gr}{\operatorname{Gr}}
\newcommand{\diag}{\operatorname{diag}}

\begin{document}
\title[Combs and complex structures on the six-sphere]{Combs and complex structures on the six-sphere}
\author{Jeff Viaclovsky}
\address{Department of Mathematics, University of California, Irvine, USA}
\email{jviaclov@uci.edu}
\date{October 7, 2026.}
\begin{abstract}
We construct complex structures on the six-sphere $S^6$ with arbitrarily large local Kuranishi dimension.
As a corollary, we find the same phenomenon holds for the smooth manifold underlying complex projective 3-space $\CP^3$.
\end{abstract}
\maketitle
\tableofcontents
\section{Introduction}
The starting point for the construction in this paper is a holomorphic
fibration $f_0 : X_0 \rightarrow \PP^1$ of Alp\"oge--Claude~\cite{AC},
where $X_0$ is diffeomorphic to the standard $S^6$.
Its critical fibers consist of two multiple fibers with bielliptic
reductions, of multiplicities $3$ and $4$, and a single fiber with
normalization $dP_6$ (a rank-two Mumford degeneration).
The original Alp\"oge--Claude manuscript was AI-generated; Philip
Engel~\cite{Engel} subsequently gave a clear, self-contained geometric
proof, which he independently verified.
Other related constructions and further examples are given in \cite{Engel,Explorer,LR26,Via}, but for simplicity we will not consider those in this paper. 
For the history of the problem of complex structures on $S^6$
and an introduction to earlier work, see \cite{ABGKR} and the
other articles in the same volume.

As described in \cite{Engel}, the Alp\"oge--Claude examples are
constructed by taking a rational elliptic surface
$\pi : S \rightarrow \PP^1$ with singular fibers of
Kodaira types $\mathrm{IV}^*\,\mathrm{III}\,\mathrm{I}_1$~\cite{KodairaII} over the critical set $\Sigma \subset \PP^1$, and Mordell--Weil group $\ZZ$.
One then removes the zero section from a
certain line bundle $M$ over $S \setminus \pi^{-1}(\Sigma)$, and then
takes a $\ZZ$-quotient generated by a suitable lift of a translation
by a narrow section. This can be filled in over the missing base values,
and then one applies log transforms at the two finite-monodromy values.
Determining that the resulting manifold is a simply connected integral
homology $6$-sphere requires difficult computations.

In this paper, we take one of these complex structures $X_0$ as our starting point. 
We will then perform a modification to $X_0$ to obtain a new complex manifold $X_m$ for every sufficiently large integer $m$ which is diffeomorphic to $X_0$, and fibers over $\PP^1$ with the same types of singular fibers and monodromy structure as the original $X_0$. The complex structure on the fibers is what changes. Importantly, we do not have to carry out any difficult computations
of integral homology groups. We write $\operatorname{Def}(X)$ for the Kuranishi germ of a compact complex manifold $X$, namely the base germ of a semiuniversal deformation of $X$. A more precise statement of our construction is as follows.

\begin{theorem} 
\label{t:main}
Let $f_0:X_0\to\PP^1$ be a fibration arising from the Alp\"oge--Claude $\mathrm{IV}^*\,\mathrm{III}\,\mathrm{I}_1$ configuration \cite{AC,Engel}. For every integer $m\geq12$,  there is a holomorphic fibration
$f_m:X_m\to\PP^1$ such that
\begin{enumerate}[label=\textup{(\roman*)}]
\item There is a diffeomorphism $\Phi_m:X_0\to X_m$ with
$f_m\circ\Phi_m=f_0$. The integral monodromy, critical values,
multiplicities, and reduced critical fibers are unchanged.
\item The fibration $f_m$ is the algebraic reduction of $X_m$; in particular, $a(X_m)=1$. The Kodaira dimension $\kappa(X_m) = - \infty$.
\item $X_m$ satisfies
\begin{align}\label{newhodge}
 h^{0,1}(X_m)=2m+1,\quad h^{0,2}(X_m)=2m,\quad h^{0,3}(X_m)=0.
\end{align}
\item The local Kuranishi dimension of $X_m$ satisfies
\begin{align}\label{intromoduli}
 h^1(X_m,\Theta_{X_m})\ \geq\
 \dim\operatorname{Def}(X_m)\ \geq\ 8m -8.
\end{align}
\end{enumerate}
\end{theorem}
Yau conjectured that every smooth manifold of real dimension $2n$,
$n>2$, admitting an almost complex structure also admits an
integrable complex structure \cite[Problem~6, p.~269]{YauReview}.
The construction of Alp\"oge--Claude establishes the $S^6$ case of
this conjecture. Theorem~1.1 gives further examples of complex structures on
$S^6$ with arbitrarily large local Kuranishi dimension.

The observation that blowing up a complex six-sphere at a point
yields a non-standard complex structure on $\PP^3$ goes back to
Hirzebruch \cite[p.~223]{Hirzebruch1954}; see also \cite{HKP,Tosatti}.
Here we show that the deformations in Theorem \ref{t:main} also lift, together with at least
one additional deformation parameter arising from the freedom to
move the point of the blow-up. The following shows that the
classification of complex
structures on $\PP^3$, as requested in Hirzebruch's Problem~28
\cite{Hirzebruch1954}, must accommodate local deformation spaces of
arbitrarily large dimension.

\begin{corollary}\label{cpthree}
The smooth manifold $\CP^3$ admits complex structures with arbitrarily
large local Kuranishi dimension. More precisely, for any point
$p_m\in X_m$, the blow-up $\widehat X_m=\operatorname{Bl}_{p_m}X_m$
is diffeomorphic to $\CP^3$ and satisfies
\begin{align}\label{cpmoduli}
 \dim\operatorname{Def}(\widehat X_m)
 \geq \dim\operatorname{Def}(X_m)+1
 \geq  8m-7.
\end{align}
These complex manifolds are non-K\"ahler, have algebraic dimension one, and have Kodaira
dimension $-\infty$. Moreover,
\begin{align}\label{cphodge}
 h^{0,1}(\widehat X_m)=2m+1,\qquad
 h^{0,2}(\widehat X_m)=2m,\qquad
 h^{0,3}(\widehat X_m)=0.
\end{align}
\end{corollary}
We do not claim the lower bounds in \eqref{cpmoduli} are optimal. It is straightforward to show directly that the complex structures in Corollary~1.2 are
non-K\"ahler, but this can also be seen by recalling a classical theorem
of Hirzebruch--Kodaira \cite{HirzKod} and Yau \cite{YauCalabi}, which implies
that any K\"ahler manifold diffeomorphic to $\PP^n$ is biholomorphic
to $\PP^n$; see also \cite{Tosatti}.

\subsection{Outline}
In Section~\ref{ray}, we 
prove that if $f : X \rightarrow \PP^1$ with $X$ diffeomorphic to $S^6$ and generic fiber a complex 2-torus, then it is cohomologically flat in all degrees, and therefore
 \begin{align}
R^1f_* \cO_{X} = \cO_{\PP^1}(a)\oplus\cO_{\PP^1}(b),\quad a\leq b, \qquad  R^2f_*\cO_{X}=\cO_{\PP^1}(c),\quad c=a+b.
\end{align}
In Section~\ref{coherent}, we prove Lemma~\ref{nowhere} about the existence of a nowhere-vanishing section of $R^1f_* \cO_X$. 
Then we show that $c\geq-1$ and that there are precisely two alternatives: $a,b\geq0$, or $(a,b)=(-1,0)$.  The
Alp\"oge--Claude configuration is of the second type. The remainder of the paper explores the first possibility.

The idea of the proof of Theorem \ref{t:main} is as follows. Denote $B = \PP^1$.  Since $f_0 : X_0 \rightarrow B$
is a torus fibration, there is a rank~$4$ flat holomorphic bundle $\mathcal{H}'$ on $B' = B \setminus \Sigma$, where $\Sigma$ denotes the critical values. The complex structure on the fiber tori determines a rank~$2$ holomorphic Hodge subbundle $\mathcal{F}'_0 \subset \mathcal{H}'$, which we view as a section $\sigma_0' : B' \rightarrow Y'$, 
where $Y'$ is a bundle with fiber the Grassmannian $\operatorname{Gr}(2,\CC^4)$. Let $D' \subset Y'$ denote the open period domain bundle. Then $\sigma_0'$ takes values in $D'$. 
The local filling models are recalled in Subsection~\ref{localmodels}, and
in Subsection~\ref{ss:period}, we show that 
the bundles $\mathcal{H}', \mathcal{F}'_0, Y'$, and section $\sigma_0'$ extend 
to bundles $\mathcal{H}, \mathcal{F}_0, Y$, and section $\sigma_0$, respectively, defined over all of $\PP^1$. We also show that in these extensions, the Gauss--Manin
connection has at most logarithmic poles along $\Sigma$.
In Subsection~\ref{tautquot}, we show that the natural tautological and quotient bundles defined over $Y'$ extend to bundles over $Y$. 

Note that the bundles $Y$ and $D' \subset Y'$ have $5$-dimensional total spaces and $4$-dimensional fibers; we are not considering polarized complex structures since the very general fiber of $f_0 : X_0 \rightarrow \PP^1$ is not polarized.
Since every complex two-torus admits a compatible flat
hyperk\"ahler metric, there are holomorphic rational curves
in the period domain $D'_p$ called \textit{twistor conics}
passing through any given point.
These have degree $2$ under the Pl\"ucker embedding of the Grassmannian into $\PP^5$. 
In Subsection \ref{conics}, we recall some basic details about twistor conics, 
and compute the restriction of the tautological and quotient bundles to a twistor conic.
We then introduce \textit{combs} in Definition~\ref{d:comb}, and give a fundamental degree calculation. More specifically, the \textit{comb} is the nodal curve given by $\sigma_0(B) \cup C_1 \cup \cdots \cup C_m$, where the \textit{teeth} $C_i$ are small complex deformations of twistor conics in $Y_{b_i}$ over distinct points $b_i \in B', i = 1, \dots, m$. 
The degree calculation shows that the degree gets larger as the number of teeth is increased; this is one of the key points. 
In Subsection \ref{combsmooth}, we next use a classical technique called \textit{comb smoothing};
see \cite{Kollar, GHS, HT}. 
In Proposition~\ref{comb}, we show that, for $m$ sufficiently large, this singular curve admits smoothings, which yield sections $\sigma_s : B \rightarrow Y$ for $s$ sufficiently small and nonzero. 

For every fixed compact subset $K\Subset B'$, these sections
define a smooth complex torus fibration over a neighborhood
of $K$ for $s$ sufficiently small.
 In Section~\ref{fillings}, we show that this torus fibration can be filled in locally in the same way as the original $X_0$ in a neighborhood of each critical value. In Subsection~\ref{compare}, we also identify the first higher direct image sheaf $R^1 f_{s*} \OO_{X_s}$.

Then in Section~\ref{glue}, we show that these local fillings can be chosen consistently to produce a global fibration over $\PP^1$ with the same type of singular fibers. 
For $s \neq 0$ sufficiently small, this yields the complex manifold $X_s$, together with a fibration $f_s : X_s \rightarrow \PP^1$, 
and the diffeomorphism $\Phi_s:X_0\to X_s$ with $f_s\circ\Phi_s=f_0$. 

In Section~\ref{s:deformations}, we count the free parameters in this construction and show that the corresponding deformations are nontrivial in the Kuranishi space. The parameters come from the freedom to choose the teeth of the comb. Note that the $X_s$ are not  small deformations of the original $X_0$, due to the singularities of the comb configuration. 

In Section~\ref{applications}, we prove Theorem~\ref{t:main} and
Corollary~\ref{cpthree}. The first step is done in
Subsection~\ref{boundexplicit} by using properties of the
Alp\"oge--Claude configuration to compute the
degree of $R^1f_{m*}\OO_{X_m}$. In Subsection~\ref{smoothingbound},
we then estimate the number of teeth of the comb sufficient for the
construction to work, which explains the restriction $m\geq12$ in Theorem~\ref{t:main}. In Subsection~\ref{nonnegative}, we derive the
Hodge numbers and prove the deformation claims. In
Subsection~\ref{analyticprops}, we discuss some other analytic
properties of $f_m:X_m\to\PP^1$. In particular, we show that $f_m$
is the algebraic reduction and determine the Kodaira dimension. 
In Subsection~\ref{blowups}, we prove
Corollary~\ref{cpthree}.

Finally, in Section~\ref{future}, we conclude the paper by listing several questions for future research.

\subsection{AI Statement}
The main mathematical ideas in this paper are due to the author
and were explored and developed further through an extended
dialogue with ChatGPT 6 Astra. The tool was used throughout the
preparation of the manuscript, both in improving
the exposition and in developing some of the mathematical
arguments. The author takes full responsibility for the
mathematical content of the paper.

\subsection{Acknowledgements}
The author would like to thank Brian Conrad for assistance with Lemma~\ref{bcgeneral} and for communicating an outline of a proof of Theorem~\ref{raynaud} from Ofer Gabber.
Thanks are also due to Nobuhiro Honda for discussions
about the material in Section~\ref{coherent} and many other crucial discussions regarding the possibility of a complex structure on $S^6$; especially through our joint work~\cite{HV}. 
Ronan Conlon, Sean Paul and Yanir Rubinstein provided very useful feedback about some complex geometry issues.  
The author also held useful discussions with Philip Engel, Wenfei Liu, and S\"onke Rollenske regarding the properties of complex structures on $S^6$.

%%%%%%%%%%%%%%%%%%%%%%%%%%%%%%%%%%%%%%%%%%%%%%%%%%%%%%%%%%55
\section{Cohomological flatness}\label{ray}
Let $f:X\to\Delta$ be a proper surjective holomorphic map from a
connected complex $n$-manifold, with connected fibers. The fiber $X_y$
means the full analytic fiber, including its multiplicities. 
We give the characteristic-zero analytic
argument corresponding to Raynaud's degree-zero cohomological flatness theorem
\cite{Raynaud}; see also \cite{GabberLodh} for higher-dimensional
algebraic versions. We begin with the following lemma. 
\begin{lemma}\label{bcgeneral}
For every $y\in\Delta$, every $j\geq0$, and every local holomorphic
coordinate $t$ on $\Delta$ with $t(y)=0$, there is a natural exact sequence
\begin{align}\label{bcseq}
0\longrightarrow (R^jf_*\cO_X)_y\otimes k(y)
 \xrightarrow{\beta_y^j}H^j(X_y,\cO_{X_y})
 \longrightarrow (R^{j+1}f_*\cO_X)_y[t]\longrightarrow0,
\end{align}
where
\begin{align*}
 (R^{j+1}f_*\cO_X)_y[t]
 =\{\xi\in(R^{j+1}f_*\cO_X)_y\mid t\xi=0\}.
\end{align*}
\end{lemma}
\begin{proof}
Apply $Rf_*$ to the Cartier-divisor sequence
$0\to\cO_X\xrightarrow{t}\cO_X\to\cO_{X_y}\to0$.
The relevant part of its long exact sequence is
\begin{align*}
 (R^jf_*\cO_X)_y\xrightarrow{t}(R^jf_*\cO_X)_y
 \longrightarrow H^j(X_y,\cO_{X_y})
 \longrightarrow (R^{j+1}f_*\cO_X)_y\xrightarrow{t}(R^{j+1}f_*\cO_X)_y.
\end{align*}
Taking the indicated cokernel and kernel proves the claim.
\end{proof}

Since the fibers are connected, $f_*\cO_X=\cO_\Delta$.
The higher direct images $R^jf_*\cO_X$, for $j\geq1$,
need not all be locally free in general.
The following theorem establishes local freeness in degree one.

\begin{theorem}\label{raynaud}
The sheaf $R^1f_*\cO_X$ is locally free, and
\begin{align}\label{raygoal}
H^0(X_y,\cO_{X_y})=\CC
 \quad\text{for every }y\in\Delta.
\end{align}
\end{theorem}

\begin{proof}
A base parameter acts injectively on $\cO_X$, so $f$ is flat.
First observe that if $g:Y\to\Delta_u$ is proper and has connected reduced central
fiber, $g_*\cO_Y=\cO_{\Delta_u}$, and $u$ is a non-zero-divisor on
$\cO_Y$, then pushing forward
\begin{align}\label{redseq}
 0\longrightarrow\cO_Y\xrightarrow{u}\cO_Y
 \longrightarrow\cO_{Y_0}\longrightarrow0
\end{align}
shows that multiplication by $u$ on $(R^1g_*\cO_Y)_0$ is injective:
the preceding map $\cO_{\Delta_u,0}\to H^0(Y_0,\cO_{Y_0})=\CC$ is
surjective. Thus $R^1g_*\cO_Y$ has no torsion at $0$.

To prove local freeness at an arbitrary base point, choose a local
coordinate in which this point is $0$.
Shrink the disc so that $f$ is smooth away from its center. By
Hironaka's embedded resolution theorem in the analytic category
(see \cite{AHV}), resolve the central divisor by
$\mu:\widetilde X\to X$, and write $h=f\mu$ and
$h^{-1}(0)=\sum m_iD_i$ with simple normal-crossing support.
Choose $M$ divisible by every $m_i$, make the base change $t=u^M$,
and normalize:
\begin{align*}
 Y=\operatorname{Norm}(\widetilde X\times_{\Delta_t}\Delta_u),
 \qquad g:Y\to\Delta_u,\qquad \pi:Y\to\widetilde X.
\end{align*}
The central fiber $Y_0$ is reduced. Indeed, locally the equation is
$u^M=z_1^{m_1}\cdots z_k^{m_k}$. Put $e_i=M/m_i$. Each component of
its normalization is the quotient of a polydisc by
\begin{align*}
 G_0=\{(\zeta_i):\zeta_i^{e_i}=1,\ \prod_i\zeta_i=1\},
 \qquad z_i=x_i^{e_i},\qquad u=\prod_i x_i,
\end{align*}
up to a constant root of unity in $u$. This finite normal model is
generically one-to-one onto the chosen component. If
$S=\CC\{x_1,\ldots,x_n\}$, averaging gives
\begin{align*}
 S^{G_0}/uS^{G_0}=(S/uS)^{G_0},
\end{align*}
which is reduced.
The normal space $Y$ has connected general fiber and no vertical
component. Its Stein factor is a normal curve of generic degree one
over $\Delta_u$, hence is $\Delta_u$. Thus $Y_0$ is connected and
$g_*\cO_Y=\cO_{\Delta_u}$. The observation above gives local freeness
of $\cB=R^1g_*\cO_Y$.

The deck group $G=\mu_M$ has quotient $Y/G=\widetilde X$.
Finite pushforward is exact and averaging splits off invariants, so
(see \cite[Section~5.2, p.~202]{GrothendieckTohoku})
\begin{align}\label{invariantR}
 R^1h_*\cO_{\widetilde X}
 =\bigl(b_*R^1g_*\cO_Y\bigr)^G,\qquad b(u)=u^M.
\end{align}
The averaging operator $|G|^{-1}\sum_{\gamma\in G}\gamma$ makes
the right side a direct summand of $b_*\cB$, which is locally free
because $b$ is finite and flat. Hence the invariant sheaf is locally free.
Finally, $\mu_*\cO_{\widetilde X}=\cO_X$, and the low-degree Leray
sequence gives an injection
$R^1f_*\cO_X\hookrightarrow R^1h_*\cO_{\widetilde X}$.
It is therefore torsion-free, hence locally free on the disc.
The degree-zero instance of Lemma~\ref{bcgeneral} now
identifies $H^0(X_y,\cO_{X_y})$ with $\CC$.
\end{proof}
We next obtain cohomological flatness in higher degrees after imposing restrictions on $X$ and the general fiber. 
\begin{theorem}\label{allfree}
If $f:X\to\PP^1$ is a surjective
holomorphic map with connected fibers, $X$ is diffeomorphic to $S^6$,
and the general fiber is a complex two-torus,
then the sheaves $R^jf_*\cO_X$ are locally free for $j=0,1,2$, of ranks
$1,2,1$, respectively. Furthermore, cup product induces an isomorphism
\begin{align}\label{cupiso}
 \Lambda^2(R^1f_*\cO_X)\simeq R^2f_*\cO_X .
\end{align}
Consequently, for every $y\in\PP^1$, including the critical values,
\begin{align}\label{fiberh}
 \bigl(h^0(\cO_{X_y}),h^1(\cO_{X_y}),h^2(\cO_{X_y})\bigr)
 =(1,2,1).
\end{align}
\end{theorem}

\begin{proof}
The degree-zero and degree-one claims follow from
Theorem~\ref{raynaud}.
The generic ranks of $R^1f_*\cO_X$ and $R^2f_*\cO_X$ are two and one
because the general fiber is a complex two-torus. Let $\cT$ be the torsion subsheaf of $R^2f_*\cO_X$,
$\cL=(R^2f_*\cO_X)/\cT$, and $\ell=\operatorname{length}\cT$.
The torsion-free rank-one sheaf $\cL$ is a line bundle.
On a smooth torus,
$\Lambda^2 H^1(F,\cO_F)\to H^2(F,\cO_F)$ is an isomorphism.
Thus the composite map
\begin{align}\label{cupquot}
 \det(R^1f_*\cO_X)\xrightarrow{\smile}R^2f_*\cO_X\longrightarrow\cL,
\end{align}
where the second arrow is the quotient map, is a nonzero
homomorphism of line bundles. Consequently,
$\deg\cL-\deg(R^1f_*\cO_X)\geq0$.

Hirzebruch--Riemann--Roch gives
\begin{align}
 \chi(X,\cO_X)=\frac1{24}\int_X c_1(X)c_2(X)=0,
\end{align}
since $H^2(X,\ZZ)=H^4(X,\ZZ)=0$.
The Leray spectral sequence for $\cO_X$ is
\begin{align*}
 E_2^{p,q}=H^p(\PP^1,R^qf_*\cO_X)
 \ \Longrightarrow\ H^{p+q}(X,\cO_X).
\end{align*}
There are no coherent direct images in degrees greater than two, so
\begin{align*}
 \chi(X,\cO_X)
&=\sum_{p,q}(-1)^{p+q}
   \dim_{\CC}H^p(\PP^1,R^qf_*\cO_X)\\
&=\sum_{q=0}^2(-1)^q\chi(\PP^1,R^qf_*\cO_X)
\\
&=\chi(\cO_{\PP^1})-\chi(R^1f_*\cO_X)+\chi(R^2f_*\cO_X).
\end{align*}
The exact sequence
\begin{align*}
 0\longrightarrow\cT\longrightarrow R^2f_*\cO_X
 \longrightarrow\cL\longrightarrow0
\end{align*}
gives
\begin{align*}
 \chi(R^2f_*\cO_X)=\chi(\cL)+\chi(\cT)=\chi(\cL)+\ell,
\end{align*}
since $\cT$ has finite support and
$h^0(\PP^1,\cT)=\ell$, $H^1(\PP^1,\cT)=0$.
Applying Riemann--Roch on $\PP^1$ therefore yields
\begin{align}\label{rrdegree}
\begin{split}
 0&=\chi(\cO_{\PP^1})-\chi(R^1f_*\cO_X)+\chi(R^2f_*\cO_X)\\
 &=1-(2+\deg(R^1f_*\cO_X))+(1+\deg\cL+\ell)\\
 &=\deg\cL-\deg(R^1f_*\cO_X)+\ell.
\end{split}
\end{align}
Both terms on the last line are nonnegative. Hence $\ell=0$ and
$\deg\cL=\deg(R^1f_*\cO_X)$. Since $\cT=0$, we have $R^2f_*\cO_X=\cL$, so $R^2f_*\cO_X$ is locally free.
The cup product morphism
$\Lambda^2(R^1f_*\cO_X)\to R^2f_*\cO_X$ is a nonzero homomorphism of line bundles
of the same degree, and is therefore an isomorphism.
This proves \eqref{cupiso}.

Since $R^jf_*\cO_X$ is locally free for $j=0,1,2$ and
$R^3f_*\cO_X=0$, Lemma~\ref{bcgeneral} shows that the
base-change maps are isomorphisms in degrees zero, one,
and two. The fiber cohomology dimensions are therefore the corresponding
ranks, proving \eqref{fiberh}.
\end{proof}

\begin{remark}
The general local-freeness claim in \cite[Lemma~1.5]{CDP98} is
false, as already pointed out by Lehn--Rollenske--Schinko
\cite[Remark~3.7]{LRS}. The proof of \cite[Proposition~5.1]{CDP20} is also incomplete:
in the nonreduced-fiber case it assumes $h^0(X_y,\omega_{X_y})=1$,
which is the equality to be proved. Theorem~\ref{raynaud} and the
argument above establish the local freeness needed here independently.
\end{remark}

%%%%%%%%%%%%%%%%%%%%%%%%%%%%%%%%%%
\section{The splitting types}\label{coherent}
In the setting of Section~\ref{ray}, 
we know that $R^1 f_* \mathcal{O}_X = \mathcal{O}(E)$,
where $E$ is a holomorphic vector bundle on the base. 
The following lemma relates the integral cohomology of the fibers
to holomorphic sections of $E$.
\begin{lemma}\label{nowhere}
Let $f:X\to B$ be a surjective holomorphic map with connected fibers
from a complex manifold diffeomorphic to $S^6$ to a smooth connected curve.
Then $B\simeq\PP^1$ and $H^0(B,R^1f_*\ZZ)\simeq\ZZ$.
Let $E$ be the holomorphic vector bundle with
$\cE=R^1f_*\cO_X=\cO(E)$. Every nonzero section of $R^1f_*\ZZ$
maps to a nowhere-vanishing section of $E$. In particular, $E$ admits
such a section.
\end{lemma}

\begin{proof}
The curve $B$ is compact because $f$ is surjective.
Connectedness gives $f_*\ZZ=\ZZ$.
The five-term integral Leray sequence
(see \cite[p.~554, (1.7)]{HVF}) is
\begin{align*}
0 \longrightarrow H^1(B,\ZZ) \longrightarrow H^1(X,\ZZ)
\longrightarrow H^0(B,R^1f_*\ZZ)
 \xrightarrow{d_2}H^2(B,\ZZ)\longrightarrow H^2(X,\ZZ).
\end{align*}
Since $H^1(X,\ZZ)=0$, we see that $B\simeq\PP^1$,
and since also $H^2(X,\ZZ)=0$, the map
\begin{align}\label{transgression}
 H^0(\PP^1,R^1f_*\ZZ)\xrightarrow{\ d_2\ }
 H^2(\PP^1,\ZZ)\simeq\ZZ
\end{align}
is an isomorphism.

Write $\mathcal{Z}=R^1f_*\ZZ$. Use the exponential sequence
\begin{align}\label{exp}
0\longrightarrow\ZZ_X\longrightarrow\cO_X
 \xrightarrow{\ \exp(2\pi i\,\cdot)\ }\cO_X^*
 \longrightarrow1.
\end{align}
We have $f_*\cO_X=\cO_B$ and $f_*\cO_X^*=\cO_B^*$.
For the latter claim, apply $f_*\cO_X=\cO_B$ to a
nowhere-vanishing holomorphic function and its inverse.
The exponential $\cO_B\to\cO_B^*$ is surjective as a map of
sheaves, since units admit logarithms locally on $B$.
The direct-image sequence therefore yields an injection of
sheaves of abelian groups
\begin{align}\label{ainj}
0\longrightarrow\mathcal{Z}\xrightarrow{\alpha}\cE.
\end{align}
Choose $0\ne\tau\in H^0(B,\mathcal{Z})$ and put $s=\alpha(\tau)$.

First $\tau_y\ne0$ as a germ at every $y\in B$. Otherwise $\tau$
and hence $s$ would vanish on a neighborhood of $y$.
The identity theorem for sections of the vector bundle $E$ on
connected $B$ would give $s=0$, contradicting \eqref{ainj}.

Let $t$ be a local parameter at $y$, and put
$E(y)=\cE_y/t\cE_y$.
Proper base change for the constant sheaf gives
\begin{align}\label{constantbc}
 \mathcal{Z}_y\simeq H^1(X_y,\ZZ).
\end{align}
Here the topology ignores nilpotents. Equivalently, use the
deformation-retraction neighborhoods of the fiber as in
\cite[Proposition~3.1]{HVF}. On the full analytic fiber the
exponential sequence is still exact: a unit has a logarithm locally,
including its nilpotent part; see \cite[Lemma~54.3]{KaupKaup}.
Since
$H^0(X_y,\cO_{X_y})=\CC$ and
$H^0(X_y,\cO_{X_y}^*)=\CC^*$, its cohomology sequence gives an injection
\begin{align}\label{fiberexp}
 \alpha_y^F:H^1(X_y,\ZZ)\hookrightarrow H^1(X_y,\cO_{X_y}).
\end{align}

In the following diagram, the square commutes,
the two vertical maps are injective, and the top row is exact:
\begin{equation}\label{board}
\begin{tikzcd}[ampersand replacement=\&,column sep=1.5em,row sep=1.5em]
 \cE_y\arrow[r,"t"]\&\cE_y\arrow[r,"\mathrm{res}"]\&
 H^1(X_y,\cO_{X_y})\arrow[r]\&(R^2f_*\cO_X)_y[t]\arrow[r]\&0\\
 \&\mathcal{Z}_y\arrow[u,"\alpha"]\arrow[r,"\simeq"']\&
 H^1(X_y,\ZZ)\arrow[u,"\alpha_y^F"']
\end{tikzcd}
\end{equation}

By Lemma~\ref{bcgeneral}, the restriction arrow factors as
\begin{align}
 \cE_y\longrightarrow E(y)
 \xrightarrow[\text{injective}]{\ \beta_y^1\ }
 H^1(X_y,\cO_{X_y}).
\end{align}
Thus
$\beta_y^1(s(y))=\alpha_y^F(\tau_y)$.
The right-hand side is nonzero by \eqref{constantbc},
\eqref{fiberexp}, and $\tau_y\ne0$. Hence $s(y)\ne0$.
Since $y$ was arbitrary, $s$ is nowhere vanishing.
\end{proof}

\subsection{The splitting types}
Under the assumptions of Theorem~\ref{allfree}, Birkhoff--Grothendieck gives
\begin{align}\label{splittings}
 R^1f_*\cO_X=\cO_{\PP^1}(a)\oplus\cO_{\PP^1}(b),\quad a\leq b,\qquad
 R^2f_*\cO_X=\cO_{\PP^1}(c).
\end{align}
The Leray isomorphism \eqref{transgression} and
Lemma~\ref{nowhere} give a nowhere-vanishing section $s$ of $R^1f_*\cO_X$. Wedge product with $s$ and
\eqref{cupiso} give
\begin{align}\label{bundleextension}
0\longrightarrow\cO_{\PP^1}\xrightarrow{s}R^1f_*\cO_X
 \xrightarrow{s\smile-}R^2f_*\cO_X\longrightarrow0.
\end{align}

\begin{proposition}\label{degrees}
Under the same assumptions as in Theorem~\ref{allfree},
the integers in \eqref{splittings} satisfy
\begin{align}\label{degconstraints}
 c=a+b,\qquad c\geq-1.
\end{align}
Moreover, either $0\leq a\leq b$, or $(a,b)=(-1,0)$.
There are natural isomorphisms
\begin{align}
 H^1(X,\cO_X)\simeq H^0(\PP^1,R^1f_*\cO_X), \quad
 H^2(X,\cO_X) \simeq H^0(\PP^1,R^2f_*\cO_X),
\end{align}
and
\begin{align}\label{globalhodge}
 h^{0,1}(X)=a + b +2,\qquad h^{0,2}(X)= a + b +1,\qquad h^{0,3}(X)=0.
\end{align}
\end{proposition}

\begin{proof}
Equation \eqref{cupiso} gives $c=a+b$.
The coherent Leray spectral sequence is
\begin{align*}
 E_2^{p,q}=H^p(\PP^1,R^qf_*\cO_X)
 \ \Longrightarrow\ H^{p+q}(X,\cO_X).
\end{align*}
Since the base is a curve and the direct images are coherent,
only the columns $p=0,1$ can be nonzero. Every differential from
$E_2$ onward therefore vanishes. The resulting filtration gives,
for $j\geq1$, natural exact sequences
\begin{align*}
 0\longrightarrow H^1(\PP^1,R^{j-1}f_*\cO_X)
 \longrightarrow H^j(X,\cO_X)
 \longrightarrow H^0(\PP^1,R^jf_*\cO_X)\longrightarrow0.
\end{align*}
Using $f_*\cO_X=\cO_{\PP^1}$ and $R^3f_*\cO_X=0$, we obtain the isomorphisms
\begin{align}
\label{e:lerayiso}
 H^1(X,\cO_X) \simeq H^0(\PP^1,R^1f_*\cO_X), \quad
 H^3(X,\cO_X) \simeq H^1(\PP^1,R^2f_*\cO_X),
\end{align}
and the exact sequence
\begin{align}
\label{e:lerayseq}
 0\longrightarrow H^1(\PP^1,R^1f_*\cO_X)\longrightarrow H^2(X,\cO_X)
 \longrightarrow H^0(\PP^1,R^2f_*\cO_X)\longrightarrow0.
\end{align}

For a complex structure on $S^6$, the vanishing $H^0(X,K_X)=0$
is discussed in \cite[Section~2.1]{Angella}. Serre duality therefore
gives $H^3(X,\cO_X)=0$.
The second Leray isomorphism in~\eqref{e:lerayiso} therefore gives
$H^1(\PP^1,\cO_{\PP^1}(c))=0$, equivalently
$c\geq-1$.

If $a<0$, every section of $R^1f_*\cO_X$ lies in the second summand
$\cO_{\PP^1}(b)$. A nowhere-vanishing section forces $b=0$, and
$-1\leq c=a<0$ gives $a=-1$. If $a\geq0$, both degrees are nonnegative.

In both alternatives $H^1(\PP^1,R^1f_*\cO_X)=0$, since
$H^1(\PP^1,\cO_{\PP^1}(k))=0$ for $k\geq-1$. The exact sequence \eqref{e:lerayseq}
therefore gives $H^2(X,\cO_X)\simeq H^0(\PP^1,R^2f_*\cO_X)$.
Riemann--Roch on $\PP^1$ now yields
\begin{align*}
 h^0(\PP^1,R^1f_*\cO_X)&=\chi(\PP^1,R^1f_*\cO_X)=a+b + 2,\\
 h^0(\PP^1,R^2f_*\cO_X)&=\chi(\PP^1,R^2f_*\cO_X)=a + b +1.
\end{align*}
Together with $H^3(X,\cO_X)=0$, these prove \eqref{globalhodge}.
\end{proof}

\begin{remark}
The  Alp\"oge--Claude complex structures on $S^6$ belong to the second case, $(a,b)=(-1,0)$; this is proved below in Proposition~\ref{ACHodge}. The purpose of
this paper is to explore the first case, $0\leq a\leq b$, and to
construct examples with $a+b$ arbitrarily large. Both of these cases were 
claimed to be excluded in
\cite[Theorem~2.2 and Corollary~2.3]{CDP20}; the examples below show
that their exclusion of the nonnegative splitting case is also incorrect.
\end{remark}

%%%%%%%%%%%%%%%%%%%%%%%%%%%%%%%%
\section{Period bundles}\label{s:periodbundles}

By the \emph{Hodge plane} of a complex two-torus $T$,
we mean the two-dimensional complex subspace
$H^{1,0}(T)\subset H^1(T,\CC)$ consisting of cohomology classes
represented by holomorphic $1$-forms. In this section, we describe the local
models near the critical fibers and extend the cohomology and Hodge
bundles over the whole base. We also discuss extensions of the natural tautological and quotient bundles.  

\subsection{The local fillings}\label{localmodels}
Let $f_0:X_0\to B=\PP^1$ be as in Theorem~\ref{t:main}. Let
$\Sigma\ne\varnothing$ be its set of critical values and put $B'=B\setminus\Sigma$.
For precision, the two types of analytic neighborhoods are described
below. At the hexagonal fiber we begin with multiplicative periods
and then take logarithms to describe the Hodge planes.

\smallskip
\noindent\emph{Case (i): the hexagonal translation quotient.}
The punctured fibers have the multiplicative description
\begin{align}\label{units}
 (\CC^*)^2/\langle(tu_{11}(t),u_{21}(t)),(u_{12}(t),tu_{22}(t))\rangle,
\end{align}
where the $u_{ij}$ are holomorphic units. The filling uses the
standard $\ZZ^2$-periodic unimodular triangulation of $\RR^2$.

Choose a disk $\Delta$ with coordinate $t$, centered at the
critical value. Each unit $u_{ij}$ in \eqref{units} has a
holomorphic logarithm on this disk. Set
\begin{align*}
 \ell(t)=\frac{\log t}{2\pi i},\qquad
 H(t)=\frac{1}{2\pi i}
 \begin{pmatrix}
  \log u_{11}(t)&\log u_{12}(t)\\
  \log u_{21}(t)&\log u_{22}(t)
 \end{pmatrix}=(h_{ij}(t)).
\end{align*}
Taking logarithms of the multiplicative periods gives
\begin{align}\label{tau}
 \tau(t)=\ell(t)I_2+H(t)
 =\begin{pmatrix}
 \ell(t)+h_{11}(t)&h_{12}(t)\\
 h_{21}(t)&\ell(t)+h_{22}(t)
 \end{pmatrix}.
\end{align}
The matrix $H$ is holomorphic on $\Delta$, whereas $\ell$
is single-valued only after choosing a branch on $\Delta^*$.
Writing $z_i=\exp(2\pi i v_i)$ in the multiplicative description
\eqref{units}, the angular loops give the two standard periods in
$\CC^2$, and the two multiplicative identifications give the
columns of $\tau$. Thus the marked fiber is
$\CC^2/(\ZZ^2+\tau(t)\ZZ^2)$.

Let $A_1,A_2$ be the angular homology classes, and let
$B_1,B_2$ be the classes corresponding to the columns of
$\tau$. On a simply connected part of $\Delta^*$, these give
a flat integral basis of homology. Let
$\alpha_1,\alpha_2,\beta_1,\beta_2$ be the dual cohomology basis.
The holomorphic forms $\omega_i=dv_i$ satisfy
\begin{align*}
 \int_{A_j}\omega_i=\delta_{ij},\qquad
 \int_{B_j}\omega_i=\tau_{ij}(t).
\end{align*}
Since a cohomology class is determined by these four periods,
\begin{align}\label{hodgeperiodrows}
 [\omega_i]=\alpha_i+\sum_{j=1}^2\tau_{ij}(t)\beta_j,
 \qquad i=1,2.
\end{align}
Thus the Hodge plane is the row span of $(I_2,\tau(t))$
in this cohomology basis.

Going once counterclockwise around $t=0$ changes $\ell$ to
$\ell+1$, so the transported homology basis changes by
\begin{align*}
 A_i\longmapsto A_i,\qquad B_i\longmapsto B_i+A_i.
\end{align*}
The transported dual cohomology basis therefore changes by
\begin{align*}
 \alpha_i\longmapsto\alpha_i-\beta_i,\qquad
 \beta_i\longmapsto\beta_i.
\end{align*}

The filling is the hexagonal quotient described in
Subsection~\ref{hex}. Its central fiber has normalization $dP_6$,
with opposite boundary curves identified.

\smallskip
\noindent\emph{Case (ii): the finite cyclic quotients of orders three and four.}
The multiple fibers have multiplicity $3$ or $4$.
Write each local model as the free quotient of a smooth family
$g:\widetilde X\to\Delta_w$ of two-tori by
$G=\langle\gamma\rangle$, where $|G|=n\in\{3,4\}$,
$\gamma$ covers $w\mapsto\zeta w$,
$\zeta=\exp(2\pi i/n)$, and $t=w^n$.
Let $\pi:\Delta_w\to\Delta_t$ be this map. 
The action is free on the total space, including the central torus.
Thus the quotient total space is smooth, while its central fiber
has multiplicity $n$; in the models considered here its reduction
is bielliptic. Using a flat identification of the cohomology of the
covering fibers, the action from the fiber over $w$ to the fiber over
$\zeta w$, inverse to pullback by $\gamma$, has eigenvalues
$1,1,\zeta,\zeta^{-1}$ on $H^1$ in both cases
\cite[Proposition~3.4]{Engel}.
The affine translation is part of this local model,
even though it is invisible in the action on cohomology.

\subsection{Periods}
\label{ss:period}
For the classical description of complex tori by lattices and period
matrices, see Griffiths--Harris \cite[Chapter~2, Section~6]{GH}.
Write $X'_0=f_0^{-1}(B')$ and $f'_0=f_0|_{X'_0}:X'_0\to B'$.
We use primes for the local systems, bundles, and period spaces
defined over $B'$; their extensions over $B$ will be unprimed.
Let $\mathcal{Z}'=R^1f'_{0*}\ZZ$ be the rank-four cohomology local
system on $B'$, and let
$\Lambda'=\operatorname{Hom}(\mathcal{Z}',\ZZ_{B'})$ be its dual
homology local system. Thus the fiber of $\Lambda'$ at $b$ is
$H_1(X_{0,b},\ZZ)$. Write
$\Lambda'_{\RR}=\Lambda'\otimes_{\ZZ}\RR$. Put
\begin{align*}
 \cH'=\mathcal{Z}'\otimes_{\ZZ}\OO_{B'}.
\end{align*}
This is a rank-four locally free sheaf. We use the same notation for
its associated rank-four holomorphic vector bundle. It is a flat
bundle, with its \emph{Gauss--Manin connection}
\begin{align*}
 \nabla':\cH'\longrightarrow\cH'\otimes\Omega_{B'}^1.
\end{align*}
For a holomorphic vector bundle $\cE$, write $\cE(b)$ for its
vector space fiber at $b$, rather than the stalk of its sheaf of
sections. In sheaf notation this fiber is
$\cE_b/\mathfrak m_b\cE_b$. Thus
\begin{align*}
 \cH'(b)=\mathcal{Z}'_b\otimes_{\ZZ}\CC
 =H^1(X_{0,b},\CC),\qquad b\in B'.
\end{align*}
The real local system is $\mathcal{R}'=\mathcal{Z}'\otimes_{\ZZ}\RR$.
Its fiber is $H^1(X_{0,b},\RR)$, and
$\cH'(b)=\mathcal{R}'_b\otimes_{\RR}\CC$. Complex conjugation on
this complexification fixes $\mathcal{R}'_b$ pointwise: it sends
$v\otimes z$ to $v\otimes\overline z$. The integral transition
matrices commute with conjugation, so this defines a fiberwise
conjugation independently of a local marking.

A period section is a rank-two holomorphic subbundle
$\cF'\subset\cH'$ such that
\begin{align}\label{D}
 \cF'(b)\cap\overline{\cF'(b)}=0
 \quad\hbox{in }\cH'(b),\qquad b\in B'.
\end{align}
Here the bar means the conjugation just defined. Equivalently,
$\cH'(b)=\cF'(b)\oplus\overline{\cF'(b)}$.
Define the Grassmannian bundle and its open period domain bundle by
\begin{align*}
 p':Y'&=\Gr_{B'}(2,\cH')\longrightarrow B',\\
 D'&=\bigcup_{b\in B'}D'_b\subset Y',\qquad
 D'_b=\{W\in\Gr(2,\cH'(b)):W\cap\overline W=0\}.
\end{align*}
Thus $\cF'$ is equivalently a holomorphic section
$\sigma':B'\to Y'$ with image in $D'$.
The original family determines the Hodge subbundle $\cF'_0$, with
$\cF'_0(b)=H^{1,0}(X_{0,b})$, and the corresponding section
$\sigma'_0$. The natural map from complex cohomology to
$\OO$-cohomology gives the holomorphic bundle isomorphism
\begin{align}\label{hodgequotient}
 R^1f'_{0*}\OO_{X'_0}\simeq\cH'/\cF'_0.
\end{align}
Fiberwise, this is $H^1(X_{0,b},\CC)/H^{1,0}(X_{0,b})
\simeq H^1(X_{0,b},\OO_{X_{0,b}})$.

\begin{proposition}\label{periodextension}
Under the local hypotheses of Subsection~\ref{localmodels}, the
bundle $\cH'$ extends to a rank-four holomorphic vector bundle
$\cH$ on $B$, and $\cF'_0$ extends to a rank-two holomorphic
subbundle $\cF_0\subset\cH$. The Gauss--Manin connection
$\nabla'$ extends as a meromorphic connection
\begin{align*}
 \nabla:\cH\longrightarrow\cH\otimes\Omega_B^1(\Sigma)
\end{align*}
with at most logarithmic poles along $\Sigma$. Consequently
\begin{align*}
 p:Y=\Gr_B(2,\cH)\longrightarrow B
\end{align*}
extends $p':Y'\to B'$, and $\sigma'_0$ extends to the holomorphic
section $\sigma_0:B\to Y$ defined by $\cF_0$. Furthermore, the smooth fivefold $Y$ is projective. 
\end{proposition}

Here $\Omega_B^1(\Sigma)$ allows one-forms with at most simple
poles at $\Sigma$. Thus the bundle $\cH$ is holomorphic across
$\Sigma$, while the connection is allowed to have poles there.
In a local holomorphic frame near $t=0$, it has the form
$d+A(t)\,dt/t$, with $A(t)$ holomorphic at zero; this is the
meaning of ``at most logarithmic poles.'' The proof gives these
matrices explicitly.

\begin{proof}
It suffices to construct the extensions on a small disk about each
critical value, together with their identifications over the
punctured disk.

\smallskip
\noindent\emph{Case (i): the hexagonal translation quotient.}
Use the periods and flat cohomology basis introduced in case~(i)
of Subsection~\ref{localmodels}. On the punctured disk define
\begin{align}\label{logframe}
 e_i=\alpha_i+\ell(t)\beta_i,\qquad f_i=\beta_i,
 \qquad i=1,2.
\end{align}
The two changes cancel:
$(\alpha_i-\beta_i)+(\ell+1)\beta_i
=\alpha_i+\ell\beta_i$. Hence $e_i,f_i$ are single-valued
holomorphic sections of $\cH'|_{\Delta^*}$. They form a frame,
since their change-of-basis matrix from the flat frame is triangular
with identity diagonal blocks.

We now construct the extension, rather than assuming that these
sections already have values at the origin. Take a trivial rank-four
holomorphic bundle on $\Delta$, and identify its standard frame
over $\Delta^*$ with $e_1,e_2,f_1,f_2$. Gluing by this
identification extends $\cH'$ across the origin. Denote the
resulting bundle by $\cH$, and use the same symbols $e_i,f_i$
for its holomorphic frame on the whole disk.
Since the classes $\alpha_i,\beta_i$ are flat,
\begin{align*}
 \nabla e_i=f_i\otimes\frac{dt}{2\pi i t},\qquad
 \nabla f_i=0.
\end{align*}
Thus the connection has at most a logarithmic pole at the origin.

Substituting \eqref{tau} and \eqref{logframe} into
\eqref{hodgeperiodrows} gives
\begin{align}\label{extendedhodgeframe}
 [\omega_i]=e_i+\sum_{j=1}^2 H_{ij}(t)f_j.
\end{align}
The expressions on the right are holomorphic at $t=0$. They are
linearly independent there, and at every other point, because their
$e_1,e_2$ coefficients form the identity matrix. They therefore
span a rank-two subbundle $\cF_0\subset\cH$ extending
$\cF'_0$. More explicitly, in the frame $e_1,e_2,f_1,f_2$
the quotient is the free rank-two bundle given by
\begin{align*}
 \OO_\Delta^2\oplus\OO_\Delta^2\longrightarrow\OO_\Delta^2,
 \qquad (a,b)\longmapsto b-aH(t),
\end{align*}
where $a,b$ are row vectors. Its kernel is exactly the row span
of $(I_2,H(t))$.

\smallskip
\noindent\emph{Case (ii): the finite cyclic quotients.}
We give the details for the order-three point, where $t=w^3$ and
$\zeta=\exp(2\pi i/3)$. Use the covering family $g$ from
Subsection~\ref{localmodels}. On the whole covering disk there are
a rank-four flat holomorphic bundle and a rank-two Hodge subbundle
\begin{align*}
 \widetilde{\cH}=(R^1g_*\CC)\otimes_{\CC}\OO_{\Delta_w},
 \qquad \widetilde{\cF}_0\subset\widetilde{\cH}.
\end{align*}
The affine translations act trivially on cohomology. In a flat frame
$c_1,c_2,c_3,c_4$, the generator acts on $\widetilde{\cH}$ by
\begin{align*}
 \gamma(w,z)=(\zeta w,\rho z),\qquad
 \rho=\diag(1,1,\zeta,\zeta^2).
\end{align*}
These eigenvalues follow from the monodromy calculation in
\cite[Proposition~3.4]{Engel}. Here $\rho$ maps cohomology from the
fiber over $w$ to the fiber over $\zeta w$; it is the inverse of
pullback by the deck map.

Define sheaves on $\Delta_t$ by
\begin{align}\label{invext}
 \cH=(\pi_*\widetilde{\cH})^G,\qquad
 \cF_0=(\pi_*\widetilde{\cF}_0)^G.
\end{align}
Thus sections downstairs are equivariant holomorphic sections on
the covering disk, a correspondence called \emph{descent}.
Over the punctured disk these are precisely $\cH'$ and $\cF'_0$.
The equivariance conditions on a section $\sum_j a_j(w)c_j$ are
\begin{align*}
 a_1(\zeta w)&=a_1(w),& a_2(\zeta w)&=a_2(w),\\
 a_3(\zeta w)&=\zeta a_3(w),&
 a_4(\zeta w)&=\zeta^2a_4(w).
\end{align*}
Expanding in powers of $w$, every such section has the unique form
\begin{align*}
 b_1(t)c_1+b_2(t)c_2+w b_3(t)c_3+w^2b_4(t)c_4,
\end{align*}
where the $b_j$ are holomorphic at $t=0$. Hence $\cH$ has the
holomorphic frame represented upstairs by
\begin{align}\label{finiteframe}
 s_1=c_1,\qquad s_2=c_2,\qquad s_3=wc_3,\qquad s_4=w^2c_4.
\end{align}
Although the last two representatives vanish at $w=0$, they are
free generators of the invariant module. Their classes in its
fiber modulo $t$ are therefore nonzero.

We must also check that $\cF_0$ is a subbundle at $t=0$.
The sheaf $\pi_*\OO_{\Delta_w}$ is free over $\OO_{\Delta_t}$
with basis $1,w,w^2$. After shrinking the disk, every bundle upstairs
is trivial, so its finite pushforward is locally free. Taking
$G$-invariants is exact, by the averaging operator
$(1+\gamma+\gamma^2)/3$, and gives a direct summand of each such
pushforward. Applying these operations to the Hodge exact sequence
upstairs gives
\begin{align*}
 0\longrightarrow\cF_0\longrightarrow\cH\longrightarrow
 \bigl(\pi_*(\widetilde{\cH}/\widetilde{\cF}_0)\bigr)^G
 \longrightarrow0.
\end{align*}
All three terms are locally free, of ranks $2,4,2$, respectively.
Thus $\cF_0\subset\cH$ extends the original Hodge subbundle.

Since the $c_j$ are flat and $t=w^3$, the connection in
\eqref{finiteframe} is
\begin{align*}
 \nabla s_1=\nabla s_2=0,\qquad
 \nabla s_3=\frac13s_3\otimes\frac{dt}{t},\qquad
 \nabla s_4=\frac23s_4\otimes\frac{dt}{t}.
\end{align*}
This has the required logarithmic pole.

For order four, use $t=w^4$, $\zeta=i$, and
$\rho=\diag(1,1,i,-i)$. The same argument gives the frame
\begin{align}\label{finitefour}
 s_1=c_1,\qquad s_2=c_2,\qquad s_3=wc_3,\qquad s_4=w^3c_4,
\end{align}
and the connection is
\begin{align*}
 \nabla s_1=\nabla s_2=0,\qquad
 \nabla s_3=\frac14s_3\otimes\frac{dt}{t},\qquad
 \nabla s_4=\frac34s_4\otimes\frac{dt}{t}.
\end{align*}
The local-freeness argument uses the basis $1,w,w^2,w^3$ and
the averaging operator $(1+\gamma+\gamma^2+\gamma^3)/4$.

In both cases admissibility is checked on the covering disk.
The smooth central torus satisfies
\begin{align*}
 \widetilde{\cF}_0(0)\cap
 \overline{\widetilde{\cF}_0(0)}=0.
\end{align*}
This open condition persists for the lifted deformations in
Subsection~\ref{finitefill}.

Choose the disks around the critical values disjoint. Each local
construction agrees with $(\cH',\cF'_0)$ over its punctured
disk. Gluing along these identifications gives
$\cF_0\subset\cH$ over all of $B$. Taking the relative
Grassmannian gives $Y$, and the subbundle $\cF_0$ gives the
section $\sigma_0$, as claimed.

Finally, $\cH$ is algebraic by GAGA
\cite[Proposition~18]{SerreGAGA}, since $B=\PP^1$ is projective.
The relative Pl\"ucker embedding makes $Y$ projective over $B$,
so $Y$ is projective.

\end{proof}
We next record a little more information in Case~\textup{(i)} which will be needed later. 
\begin{proposition}\label{hexmonodromy}
At a critical value $b_0$ of type~\textup{(i)}, the logarithm
$N$ of the local cohomological monodromy extends to a holomorphic
endomorphism of $\cH$ near $b_0$. Its value $N_{b_0}$ satisfies
\begin{align}\label{hextransverse}
 \ker N_{b_0}=\im N_{b_0},\qquad
 \dim_{\CC}\ker N_{b_0}=2,\qquad
 \cH(b_0)=\cF_0(b_0)\oplus\ker N_{b_0}.
\end{align}
\end{proposition}

\begin{proof}
Use the frame \eqref{logframe} constructed in case~(i) of the
proof of Proposition~\ref{periodextension}, with $t(b_0)=0$.
Let $T$ denote cohomological monodromy in the transported dual
basis $\alpha_i,\beta_i$ used there. Since $(T-I)^2=0$, its
logarithm is $N=\log T=T-I$, and
\begin{align*}
 N\alpha_i=-\beta_i,\qquad N\beta_i=0.
\end{align*}
Thus $N$ extends as a holomorphic endomorphism of $\cH|_\Delta$
with $Ne_i=-f_i$ and $Nf_i=0$. In each fiber,
\begin{align*}
 \ker N=\im N=\Span_{\CC}\{f_1,f_2\}.
\end{align*}
A vector in $\cF_0(0)$ has the form
$\sum_i a_i(e_i(0)+\sum_jH_{ij}(0)f_j(0))$.
If it lies in $\ker N$, its $e_i(0)$ coefficients vanish,
so $a_1=a_2=0$. Thus the limiting plane is transverse to
$\ker N$, and
\begin{align*}
 \cH(0)&=\cF_0(0)\oplus\ker N.\qedhere
\end{align*}
\end{proof}

Proposition~\ref{hexmonodromy} is used in
Subsection~\ref{hex} to write nearby Hodge planes as graphs and to
frame the Hodge quotient by the invariant classes $f_1,f_2$.

\subsection{The tautological and quotient bundles}
\label{tautquot}

Recall that $p':Y'=\Gr_{B'}(2,\cH')\to B'$ is the Grassmannian
bundle: a point of $Y'$ is a pair $(b,W)$, where
$W\subset\cH'(b)$ is a two-plane, and $p'(b,W)=b$.
The pullback bundle $(p')^*\cH'$ has fiber $\cH'(b)$ at this point.
Its \emph{tautological subbundle} $\mathcal{S}'$ has fiber $W$, and the
quotient bundle $\mathcal{Q}'$ has fiber $\cH'(b)/W$. These rank-two
bundles give the universal exact sequence
\begin{align}\label{univprime}
 0\longrightarrow \mathcal{S}'\longrightarrow (p')^*\cH'
 \longrightarrow \mathcal{Q}'\longrightarrow0.
\end{align}
Pulling this sequence back by a section $\sigma':B'\to Y'$
recovers the corresponding rank-two subbundle
$\cF'=(\sigma')^*\mathcal{S}'\subset\cH'$
and quotient $(\sigma')^*\mathcal{Q}'\simeq\cH'/\cF'$. In particular,
$\cF'_0=(\sigma'_0)^*\mathcal{S}'$. We next show that these bundles
and sequence \eqref{univprime} also extend to $Y$.

\begin{proposition}\label{univextension}
Under the extensions in Proposition~\ref{periodextension},
the bundles $\mathcal{S}'$ and $\mathcal{Q}'$ extend to rank-two
holomorphic bundles $\mathcal{S}$ and $\mathcal{Q}$ on $Y$,
respectively, and the exact sequence \eqref{univprime} extends to
\begin{align}\label{univ}
 0\longrightarrow \mathcal{S}\longrightarrow p^*\cH
 \longrightarrow \mathcal{Q}\longrightarrow 0.
\end{align}
Pulling back along a section $\sigma:B\to Y$ gives
\begin{align}
 0\longrightarrow\cF_\sigma\longrightarrow\cH
 \longrightarrow\cH/\cF_\sigma\longrightarrow0,
\end{align}
where $\cF_\sigma=\sigma^*\mathcal{S}$ is the rank-two subbundle
determined by $\sigma$ and
$\sigma^*\mathcal{Q}\simeq\cH/\cF_\sigma$.
\end{proposition}

\begin{proof}
Let $\mathcal{S}$ and $\mathcal{Q}$ be the tautological subbundle
and quotient bundle on $Y=\Gr_B(2,\cH)$. Locally on $B$,
a trivialization of $\cH$ identifies these with the usual
tautological and quotient bundles on $\Gr(2,\CC^4)$, giving
the exact sequence \eqref{univ}. The transition functions of
$\cH$ induce compatible identifications of these subbundles and
quotients, so the local sequences glue. Since $\cH|_{B'}=\cH'$
and $p^{-1}(B')=Y'$,
the restriction to $Y'$ is \eqref{univprime}.
Finally, \eqref{univ} is locally split because $\mathcal{Q}$
is locally free, so its pullback by $\sigma$ is exact.
The identity $p\circ\sigma=\id_B$ gives the claimed sequence on $B$.
\end{proof}
Note that the proposition applied in particular to the original section $\sigma_0$ yields
\begin{align}
\cF_0=\sigma_0^*\mathcal{S}, \quad \sigma_0^*\mathcal{Q}\simeq\cH/\cF_0
\end{align} 
over all of $B$.

%%%%%%%%%%%%%%%%%%%%%%%%%%%%%%%%%%%%%%%%%%%
\section{Combs from twistor conics} \label{s:combs}
In this section, we discuss the existence of rational curves called \textit{twistor conics}, define \textit{combs}, and then investigate their deformations. Using this, we construct new period sections from the original $2$-torus fibration. 

\subsection{Rational curves in the period domain}\label{conics}
We briefly recall how rational curves in the period domain arise from the classical twistor
construction; see \cite[14.10--14.11]{Besse}. Choose a Euclidean
inner product compatible with the given complex structure $I$
on a real four-dimensional vector space, an orthogonal complex
structure $J$ anticommuting with $I$, and put $K=IJ$.
The quaternionic relations give
\begin{align*}
 (xI+yJ+zK)^2=-\id
 \qquad\text{if }x^2+y^2+z^2=1.
\end{align*}
These complex structures descend to the quotient by any lattice,
since they are translation invariant. Identifying the parameter
sphere with $\PP^1$, their Hodge planes form a holomorphic rational
curve in the period domain, which we call a \emph{twistor conic}
since its Pl\"ucker degree is two; see
\cite[Corollaries~2.8 and~5.2]{BuskinIzadi} and the calculation below.
The following lemma states the elementary bundle and deformation
properties needed for our application.

\begin{lemma}\label{twistor}
Through any point $x\in D'_b$ there is a twistor conic
$C\subset D'_b$ such that, after identifying $C$ with $\PP^1$,
\begin{align*}
 \mathcal{S}|_C\simeq\OO_{\PP^1}(-1)^{\oplus2},\qquad
 \mathcal{Q}|_C\simeq\OO_{\PP^1}(1)^{\oplus2},\qquad
 N_{C/Y_b}\simeq\OO_{\PP^1}(2)^{\oplus3}.
\end{align*}
The local space of embedded deformations of $C$ containing $x$
is smooth of complex dimension six. Sufficiently small such
deformations are smooth conics in $D'_b$ with the same bundle
splittings, and their tangent lines at $x$ fill a nonempty
complex-open subset of $\PP(T_xY_b)$.
\end{lemma}

\begin{proof}
Choose a positive definite real inner product compatible with the
complex structure at $x$. One component of its maximal isotropic
complex two-planes is the twistor conic just described, holomorphically
embedded as a conic in $\Gr(2,\CC^4)$. A plane in this curve cannot
meet its conjugate: otherwise their intersection supplies a nonzero
real isotropic vector. In the quaternionic description the tautological
bundle is $\OO_{\PP^1}(-1)\otimes\CC^2$, and the quotient is
$\OO_{\PP^1}(1)\otimes\CC^2$. This also gives its Pl\"ucker degree $2$.

Since $T_{\Gr}=\mathcal{S}^*\otimes \mathcal{Q}$, its restriction is $\OO_{\PP^1}(2)^{\oplus4}$.
The tangent subbundle of $C$ is $\OO_{\PP^1}(2)$; its inclusion is given
by a nonzero constant vector. Hence its quotient is
$N_{C/Y_b}\simeq\OO_{\PP^1}(2)^{\oplus3}$.

Embedded deformations containing $x$ have tangent space
$H^0(C,N_{C/Y_b}(-x))$ and obstruction space
$H^1(C,N_{C/Y_b}(-x))$. The normal-bundle calculation gives
\begin{align*}
 h^0(C,N_{C/Y_b}(-x))
 &=h^0(\PP^1,\OO_{\PP^1}(1)^{\oplus3})=3\cdot2=6,\\
 H^1(C,N_{C/Y_b}(-x))&=0.
\end{align*}
Thus this deformation space is smooth of complex dimension six.
The curves remain smooth conics, and the three bundle splittings
persist under sufficiently small deformations. Since $C$ is compact
inside the open set $D'_b$, these deformations remain in $D'_b$.
Finally,
\begin{align*}
 H^1(C,N_{C/Y_b}(-2x))=0
\end{align*}
implies that normal sections vanishing at $x$ can have arbitrary
first derivative there. The tangent-direction map from the deformation
space is therefore a submersion at $C$, proving the last statement.
\end{proof}

The twistor conics through a fixed point form a real manifold
of dimension three; see \cite[Corollary~2.5]{BuskinIzadi}.
By Lemma~\ref{twistor}, allowing general complex deformations
through that point gives a smooth space of complex dimension six.
These deformed curves need not themselves be twistor conics.
Their tangent directions fill a complex-open set, providing the
freedom needed in the smoothing argument below.

\begin{definition}
\label{d:comb}
Fix $m$ distinct points $b_i\in B'$,
and attach a small complex deformation $C_i\subset Y_{b_i}$ of a twistor conic through $\sigma_0(b_i)$. The nodal curve
\begin{align*}
 C =\sigma_0(B)\cup C_1\cup\cdots\cup C_m
\end{align*}
is called a \textit{comb} of arithmetic genus zero.  The \textit{handle} is the 
section $\sigma_0(B)$ and the \textit{teeth} are the curves $C_i$. 
\end{definition}

The following calculation is the key motivation behind this definition and our main construction.  Let us suppose that $C$ could be perturbed to a smooth section 
$\sigma_s : B \rightarrow Y$ for $s \neq 0$. Then consider the rank $2$ bundles  
\begin{align}
E_s :=  \sigma_s^* \mathcal{Q}.
\end{align}
Then we have 
\begin{align}
\deg (E_s) = \int_B c_1 (\sigma_s^* \mathcal{Q}) 
= \int_{\sigma_s(B)} c_1 ( \det \mathcal{Q})
= c_1 ( \det \mathcal{Q}) \cdot [ \sigma_s(B)],
\end{align}
where $\cdot$ denotes the natural pairing between cohomology and homology. 
Since $\sigma_s(B)$ is homologous to $C$, we have 
\begin{align}
\deg (E_s) = c_1 ( \det \mathcal{Q}) \cdot  [\sigma_0(B)] +
\sum_{i = 1}^m   c_1 ( \det \mathcal{Q}) \cdot  [C_i].
\end{align}
Lemma \ref{twistor} tells us, upon identifying $C_i$ with $\PP^1$, that 
\begin{align}
\mathcal{Q}|_{C_i}\simeq \mathcal{O}_{\PP^1}(1) \oplus \mathcal{O}_{\PP^1}(1),
\quad \det(\mathcal{Q})|_{C_i} \simeq \mathcal{O}_{\PP^1}(2). 
\end{align}
We conclude that 
\begin{align}
\deg (E_s) = \deg(E_0) + 2m. 
\end{align}
Thus we can make the degree as large as we want by choosing a comb with many teeth.
An ordinary deformation of $\sigma_0$ would preserve its curve class
and hence this degree; attaching the teeth changes the class before smoothing.
To obtain the desired examples, it therefore remains to carry out three steps:
\begin{enumerate}
\item For $m$ sufficiently large, smooth the comb to sections $\sigma_s:B\to Y$ for $s$ sufficiently small and nonzero, agreeing with $\sigma_0$ at $\Sigma$.
\item
Use $\sigma_s$ to construct a smooth complex threefold $X_s$, diffeomorphic to $X_0$, with a holomorphic fibration $f_s:X_s\to B$ whose general fiber is a complex $2$-torus. 
\item Prove that $E_s\simeq R^1f_{s*}\mathcal{O}_{X_s}$ over all of $B$, including the critical values, so that no additional degree corrections occur there. 
\end{enumerate}
These steps will give examples in the first case of Proposition~\ref{degrees}
with $a+b>0$ for $m$ sufficiently large. 
The first step will be dealt with in the next subsection; the local fillings
and bundle comparison are proved in Section~\ref{fillings}, and the global
construction and diffeomorphism in Section~\ref{glue}. 
\subsection{Smoothing the comb}
\label{combsmooth}

The following proposition applies a standard comb-smoothing
argument to our period bundle; see
\cite[Section~II.7]{Kollar} and \cite[Section~2.1]{GHS}.
We use the normal-sheaf formulation with fixed marked points
described in \cite[Lemma~20 and Proposition~24]{HT}.
The resulting section agrees with $\sigma_0$ at every critical
value in $\Sigma$.

\begin{proposition}\label{comb}
For all sufficiently large $m$, the teeth can be chosen so that $C$
has a deformation to sections $\sigma_s:B\to Y$, $s\neq0$,
agreeing with $\sigma_0$ at $\Sigma$. Their curve class is
\begin{align*}
 [\sigma_s(B)]=[C]=[\sigma_0(B)]+\sum_{i=1}^m[C_i]
 \quad\text{in }H_2(Y,\ZZ).
\end{align*}
For every compact subset $K\subset B'$, the images $\sigma_s(K)$ lie
in the admissible period domain for all sufficiently small $s\neq0$,
where the bound on $|s|$ may depend on $K$.
The smoothings can also be chosen so that
\begin{align}\label{normalvan}
 H^1\bigl(B,N_{\sigma_s(B)/Y}(-\Sigma)\bigr)=0,
\end{align}
where $\Sigma$ is considered as a reduced divisor.
\end{proposition}

\begin{proof}
The nodal comb is a local complete intersection in $Y$. Its normal
sheaf is the locally free sheaf
$N_{C/Y}=\sHom(\mathcal{I}_C/\mathcal{I}_C^2,\mathcal{O}_C)$,
where $\mathcal{I}_C$ is its ideal sheaf.
For the embedded deformation theory used here, see
\cite[Sections~2,~6, and~9]{HartshorneDef}.
Write
\begin{align*}
 N_{\sigma_0(B)/Y}\simeq\bigoplus_{j=1}^4\OO_{\PP^1}(n_j),
 \qquad r=\deg\Sigma=|\Sigma|.
\end{align*}
The restriction of $N_{C/Y}$ to the handle fits into the exact sequence
\cite[Section~4.3]{HT}
\begin{align}\label{mod}
 0\longrightarrow N_{\sigma_0(B)/Y}
 \longrightarrow N_{C/Y}|_B
 \longrightarrow\bigoplus_{i=1}^m\CC_{b_i}\longrightarrow0.
\end{align}
Here each $\CC_{b_i}$ is a one-dimensional sheaf supported at the node.
The middle bundle allows a simple pole at $b_i$ in the normal
direction given by the tangent to $C_i$; this is called a positive
elementary modification. In that direction a local generator $e$
is replaced by $e/t$, where $t$ vanishes at $b_i$.
For example, choosing
$m_j\geq\max(0,r-n_j)$ modification directions in the $j$-th split
summand makes the modified bundle, after tensoring with $\OO_{\PP^1}(-\Sigma)$,
globally generated with vanishing $H^1$. Thus a sufficient bound for
generic modification directions is
\begin{align}\label{bound}
 m\geq\sum_{j=1}^4\max(0,r-n_j).
\end{align}
Let $N_0=N_{\sigma_0(B)/Y}$. The choices of modification
directions form $\prod_{i=1}^m\PP(N_0(b_i))$, using the
identification $N_0(b_i)\simeq T_{\sigma_0(b_i)}Y_{b_i}$.
The choices for which the modified bundle, after twisting by
$-\Sigma$, is globally generated with vanishing $H^1$ form
a Zariski-open subset of this product. The preceding example
shows that this subset is nonempty. Since a proper algebraic subset
has empty complex-analytic interior, this Zariski-open subset meets
the product of the nonempty complex-open sets furnished by
Lemma~\ref{twistor}.
Hence the modifications can be realized by conics with image
entirely inside the period domain. 

On each tooth, identified with $\PP^1$, there is an exact sequence
\begin{align*}
 0\longrightarrow\OO_{\PP^1}(2)^{\oplus3}
 \longrightarrow N_{C/Y}|_{C_i}
 \longrightarrow\OO_{\PP^1}(1)\longrightarrow0 .
\end{align*}
It splits, since $H^1(\OO_{\PP^1}(1))=0$. The $\OO_{\PP^1}(1)$ is the modification
at the node of the trivial direction normal to the Grassmannian fiber.
Here $\Sigma$ on $C$ denotes the divisor of smooth handle points
$\sigma_0(\Sigma)$, disjoint from all nodes.
Therefore $N_{C/Y}(-\Sigma)$ is globally generated and has vanishing $H^1$
on every component, and hence on the comb $C$ by
\cite[Lemma~21]{HT}.

Embedded deformations containing the fixed points of the handle over
$\Sigma$ are unobstructed. Global generation allows a
normal section whose image in each of the finitely many node-smoothing
spaces is nonzero. Integrating such a deformation smooths every node
and retains the prescribed values. This is the usual comb-smoothing
argument; compare \cite[Lemma~22 and Proposition~24]{HT}.

These flat embedded deformations preserve the class of the entire
comb in $H_2(Y,\ZZ)$. The smoothed curve has genus zero and degree
one over $B$, so it is the image of a section $\sigma_s$ with the
stated curve class. Since it contains every fixed point
$\sigma_0(b)$ for $b\in\Sigma$, this section satisfies
$\sigma_s(b)=\sigma_0(b)$ there.
Near each point of the handle away from the
nodes, the family is holomorphic in $(b,s)$, including $s=0$.
Finally, the portion of the entire comb over a fixed compact subset
of $B'$ is compact inside the open period domain. Sufficiently small
deformations remain in that domain there.

To prove \eqref{normalvan}, we show that the normal-bundle vanishing persists under smoothing. The nodal comb is a local complete intersection in the
smooth variety $Y$. After shrinking the smoothing base, its relative
conormal sheaf is locally free and commutes with restriction to each
fiber, and the same holds for its dual. The fixed points over $\Sigma$
are disjoint from the nodes and give a relative Cartier divisor
on the family of curves. Twisting the relative normal bundle by the
negative of this divisor gives a sheaf flat over the smoothing base.
Semicontinuity, applied to the vanishing on $C$ proved above, gives
\eqref{normalvan}.
\end{proof}

\begin{remark}\label{combconv}
On every compact subset of $B\setminus\{b_1,\ldots,b_m\}$, the sections
$\sigma_s$ converge to $\sigma_0$ in $C^\infty$ as $s\to0$.
In local holomorphic coordinates, their component functions and all
derivatives converge uniformly. This follows from the holomorphic
graph description near the handle away from the nodes. At the $b_i$,
the limiting curve contains the entire tooth $C_i$, so this convergence
to the original section does not hold there. We will next prove that $\sigma_s(B')\subset D'$ over all of $B'$ for sufficiently small $s\ne0$, using particular properties of the local fillings. 
\end{remark}

%%%%%%%%%%%%%%%%%%%%%%%%%%%%%%%%%%%%%%%%%%%%
\section{Local fillings}\label{fillings}
Near each critical value, the period sections obtained in
Section~\ref{s:combs} form a holomorphic family $\sigma_s$ with
$\sigma_s(0)=\sigma_0(0)$. Write $\cF_s=\sigma_s^*\mathcal{S}$.
We use the same local fillings as in the original construction
\cite{Mumford,Engel}, in the order of the two cases in
Subsection~\ref{localmodels}. We first recall each filling, then
prove that it extends to the perturbed periods and retains the
required smooth attachments. Each local proposition also identifies
the first coherent direct image with $\cH/\cF_s$.
Subsection~\ref{compare} shows the compatibility of these
identifications for the global construction.

\subsection{The hexagonal filling}\label{hex}

We recall the toric model used in the hexagonal filling
\cite[Proposition~2.8]{Engel}. Start with the square lattice
$\ZZ^2\subset\RR^2$, and divide every unit square by the
diagonal from $(a,b)$ to $(a+1,b+1)$. For each triangle with
vertices $v_1,v_2,v_3$, take the cone in $\RR^3$ generated by
$(v_1,1),(v_2,1),(v_3,1)$. These cones and their faces form a fan.
Each triple of generators is a basis of $\ZZ^3$, so the toric
charts are copies of $\CC^3$, glued by the corresponding
monomial coordinate changes. The resulting toric manifold has a
holomorphic function $t$ whose expression in every such chart is
\begin{align}\label{semistable}
 t=x_1x_2x_3.
\end{align}
Let $\widetilde{\mathcal X}$ denote its restriction to a small
disk $\Delta_t$. Over $\Delta_t^*$ it is
$(\CC^*)^2\times\Delta_t^*$, with coordinates $(z_1,z_2,t)$.

The central components are indexed by the lattice vertices, double
curves by edges, and triple points by triangles. Each component is
$dP_6$: its fan has rays $\pm(1,0),\pm(0,1),\pm(1,1)$.
Quotienting by
the two lattice translations, with the torus multiplications
specified by the units in the periods, gives the hexagonal filling.
Its central fiber has normalization $dP_6$, with opposite boundary
curves identified. The following proposition verifies that this
construction works for the deformed periods as well.

\begin{proposition}\label{hexfill}
In case~\textup{(i)}, let $\sigma_s$ be a holomorphic family of
local period sections with $\sigma_s(0)=\sigma_0(0)$.
After shrinking the base and parameter disks, the corresponding
torus families over $\Delta_t^*$ admit a proper flat filling
\begin{align*}
 \mathcal X_{\rm hex}\longrightarrow\Delta_t\times\Delta_s
\end{align*}
with smooth total space and submersive projection to $\Delta_s$.
For each $s$, the map $f_s:X_s\to\Delta_t$ has the original
hexagonal central fiber. The fillings are diffeomorphic over the
base, compatibly with the smooth boundary attachments. Moreover,
the natural Hodge quotient identification over $\Delta_t^*$
extends to an isomorphism
\begin{align}\label{heximage}
 R^1f_{s*}\OO_{X_s}\simeq\cH/\cF_s=\sigma_s^*\mathcal{Q}
 \qquad\hbox{on }\Delta_t.
\end{align}
\end{proposition}
\begin{proof}
At a hexagonal point, the transversality in
Proposition~\ref{hexmonodromy}, equation~\eqref{hextransverse},
persists for nearby sections. In the frame \eqref{logframe},
projection of each such plane onto $\Span_{\CC}\{e_1,e_2\}$
along $\ker N=\Span_{\CC}\{f_1,f_2\}$ is an isomorphism.
Each plane is therefore the graph of a unique holomorphic matrix
$H_s(t)$.
In this frame the plane has matrix $(I,H_s(t))$. Returning to
the flat basis of Subsection~\ref{localmodels}, its period matrix is
\begin{align*}
 \tau_s(t)=\frac{\log t}{2\pi i}I+H_s(t).
\end{align*}
The matrices $H_s$ are bounded uniformly for small $s,t$.
Hence
\begin{align*}
 \Im\tau_s(t)=-\frac{\log|t|}{2\pi}I+\Im H_s(t)
\end{align*}
is invertible for sufficiently small $t\neq0$, uniformly for small $s$.

Set $u_{ij}(t,s)=\exp(2\pi i H_{ij,s}(t))$. On the torus
$(\CC^*)^2$ take the commuting transformations
\begin{align}
 g_1(z_1,z_2,t,s)&=
 (t u_{11}(t,s)z_1,\ u_{21}(t,s)z_2,\ t,s),\label{g1}\\
 g_2(z_1,z_2,t,s)&=
 (u_{12}(t,s)z_1,\ t u_{22}(t,s)z_2,\ t,s).\label{g2}
\end{align}
The monomial maps $(z_1,z_2,t)\mapsto(tz_1,z_2,t)$ and
$(z_1,z_2,t)\mapsto(z_1,tz_2,t)$ induce the lattice maps
\begin{align*}
 (a,b,c)\longmapsto(a+c,b,c),\qquad
 (a,b,c)\longmapsto(a,b+c,c).
\end{align*}
On the plane $c=1$, these translate the triangulation by $(1,0)$
and $(0,1)$, respectively. They preserve the fan and therefore extend
to automorphisms of $\widetilde{\mathcal X}$. Multiplication by the
units $u_{ij}(t,s)$ also extends, by the torus action on this toric
manifold. These maps preserve $(t,s)$, so \eqref{g1}--\eqref{g2}
extend to commuting automorphisms of
$\widetilde{\mathcal X}\times\Delta_s$.

We check that their action is free and properly discontinuous, and
that the quotient is proper over $\Delta_t\times\Delta_s$.
For a point with $t\neq0$, put
\begin{align*}
 r(z,t)=\frac{(\log|z_1|,\log|z_2|)}{\log|t|}\in\RR^2.
\end{align*}
In the chart corresponding to a triangle with vertices $v_1,v_2,v_3$,
\begin{align*}
 r=\sum_{j=1}^3
 \frac{\log|x_j|}{\log|t|}\,v_j,\qquad
 \sum_{j=1}^3\frac{\log|x_j|}{\log|t|}=1.
\end{align*}
On a compact subset of the chart the $|x_j|$ are bounded above.
As $t\to0$, the displayed coefficients therefore have a uniform
lower bound tending to zero, and consequently a uniform upper bound.
It follows that $r$ stays bounded on a fixed compact subset of
$\widetilde{\mathcal X}$ as $t\to0$.

The action of $\lambda\in\ZZ^2$ on these real coordinates is
\begin{align}\label{tropicalaction}
 r\longmapsto r+A(t,s)\lambda,\qquad
 A(t,s)=I+\frac{U(t,s)}{\log|t|},
\end{align}
where $U$ is the bounded real matrix of logarithms of the absolute
values of the units in \eqref{g1}--\eqref{g2}. After shrinking the
disk, $A$ and $A^{-1}$ are uniformly bounded.
If a sequence of group translates carries a fixed compact set back
to meet itself at points with $t\to0$, boundedness of $r$ and
\eqref{tropicalaction} force the corresponding $\lambda$'s to be
bounded. Away from $t=0$ the same conclusion is the usual properness
of a full lattice action on logarithmic radii.

On the central fiber, strata are indexed by the cells of the
triangulation. A nonzero lattice translation fixes no cell; compact
sets meet only finitely many central strata. This proves freeness
there and gives the analogous proper-discontinuity claim on the
central fiber. On a noncentral fiber, $A\lambda=0$ implies
$\lambda=0$, proving freeness there as well.

To check properness of the quotient over the base, translate $r$
into a fundamental parallelogram for the lattice $A(t,s)\ZZ^2$.
These parallelograms lie in a fixed bounded region. Only finitely many
triangles of the fixed triangulation meet that region. Choosing a
triangle containing the translated $r$, its barycentric coordinates
are nonnegative, so the corresponding toric coordinates satisfy
$|x_j|\leq1$. Intersect these finitely many closed unit polydisks
with the inverse image of a smaller closed $t$-disk, and let $s$
range over a smaller closed parameter disk. This gives compact
representatives for all noncentral fibers over these disks.
Representatives for the central quotient are compact as well:
one fundamental collection of toric components and their boundary
strata suffices.

A free properly discontinuous action on this Hausdorff toric manifold
has a Hausdorff manifold quotient. This gives a proper analytic quotient with smooth total space
\begin{align}\label{hexquotient}
 \mathcal X_{\rm hex}
 =(\widetilde{\mathcal X}\times\Delta_s)/\langle g_1,g_2\rangle
 \longrightarrow\Delta_t\times\Delta_s .
\end{align}
Its local map to the base is $(x_1,x_2,x_3,s)\mapsto
(x_1x_2x_3,s)$, or a lower-codimension version of this expression.
It is therefore flat, and its projection to $\Delta_s$ is a
submersion. Over $t\neq0$, taking logarithms in
\eqref{g1}--\eqref{g2} recovers the periods $\tau_s$.

Since the period section is fixed at $t=0$, $H_s(0)=H_0(0)$.
The unit values defining the central gluings therefore agree exactly,
so the central fiber is unchanged.
In local quotient charts the equation for $t$ is independent of $s$.
Local lifts of a vector field in the $s$-direction can consequently
be chosen to annihilate $t$, then patched by a smooth partition of
unity. On compact inverse images of closed $t$-disks, their flows
give diffeomorphisms of the fillings preserving the map to $t$.
On a regular boundary collar these lifts can be chosen to agree with
the prescribed smooth identifications of the torus families.

It remains to prove \eqref{heximage}. Fix $s$ and write
$f:X\to\Delta$ and $\cF=\cF_s$. The transversality used above shows that $\ker N$ maps isomorphically to
$\cH/\cF$. Its constant sections therefore give a frame for
this quotient.
We compare them with coherent cohomology at the central fiber $R$.

The normalization of $R$ is $dP_6$, with three normalized double
curves $\PP^1$ and two triple points. Choose the three oriented
edges of the unit-square triangulation to be $a,b,c$, where $c$
is the diagonal. The degree-zero rows in both the constant and
holomorphic normalization resolutions are
\begin{align}\label{normalizationmatrix}
 \CC\xrightarrow{\,0\,}\CC^3
 \xrightarrow{\ D\ }\CC^2,\qquad
 D=\begin{pmatrix}1&1&-1\\-1&-1&1\end{pmatrix}.
\end{align}
The normalization strata have vanishing first singular cohomology and
vanishing higher $\OO$-cohomology. Thus $H^1(R,\OO_R)=\ker D$ has dimension
two, $H^2(R,\OO_R)=\coker D$ has dimension one, and the natural map
in degree one is an isomorphism:
\begin{align}\label{centralcomparison}
 H^1(R,\CC)\xrightarrow{\ \sim\ }H^1(R,\OO_R).
\end{align}

Use the homology basis $(A_1,A_2,B_1,B_2)$ and dual
cohomology basis of case~(i) in Subsection~\ref{localmodels}.
The invariant integral cohomology classes are generated by
$\beta_1,\beta_2$. We construct their extensions over the filled
neighborhood directly from the covering space.

The covering $\widetilde{\mathcal X}\to X$ of the entire local filling
has deck group $\ZZ^2$. Its two characters define integral
cohomology classes on the entire neighborhood. On a smooth fiber
they are exactly $\beta_1,\beta_2$: angular loops lift to closed
loops, and the two $B_i$ lift with the corresponding deck
translations. On the central fiber the cover is the infinite quilt
of rational toric surfaces. Its components and intersections are
simply connected and its dual complex is the plane; van Kampen gives
simple connectedness of this quilt. Thus the characters restrict to
an integral basis of $H^1(R,\ZZ)$.
Equivalently, in \eqref{normalizationmatrix} their representatives
are $(1,0,1)$ and $(0,1,1)$.

This constructs the extending integral classes directly. It also
proves that specialization identifies $H^1(R,\ZZ)$ with the full
primitive invariant lattice. Since $f_i=\beta_i$ in the frame
\eqref{logframe}, these are
exactly the classes whose images frame $\cH/\cF$. The same
statement holds for each nearby period section by the transversality
used above.

The map from the constant sheaf to the structure sheaf sends these
classes to holomorphic sections of $R^1f_*\OO_X$.
By Grauert, applied to $h^1(\OO_R)=2$ and the nearby tori,
\eqref{centralcomparison} shows that these two invariant classes
give a coherent frame at $t=0$. On the punctured disk that frame is
the image of the same classes in the Hodge quotient. Thus the
isomorphism over the punctured disk extends across the origin,
without a zero or pole in its determinant.
\end{proof}

\subsection{Finite quotients}\label{finitefill}

In case~\textup{(ii)}, the original filling is the quotient of a
smooth family of tori over $\Delta_w$ by a free cyclic group
$G$ of order $n=3$ or $4$, acting affinely on the fibers and
covering $w\mapsto\zeta w$, where $\zeta=\exp(2\pi i/n)$.
The quotient map on the base is $\pi:\Delta_w\to\Delta_t$,
$t=w^n$. The central fiber downstairs has multiplicity $n$,
and its reduction is the quotient of the central torus by $G$.
To deform this filling, we lift the new Hodge planes to the whole
covering disk and retain the translation part of the affine action.

After changing the origin by a holomorphic section, the original
generator can be written in real torus coordinates as
\begin{align}\label{affine}
 (w,x)\longmapsto(\zeta w,T x+v/n),
 \qquad T^n=I,\quad v\in\widetilde\Lambda^T,\quad n\in\{3,4\}.
\end{align}
Here $\widetilde\Lambda$ is the homology lattice of the smooth
covering family, extending the pullback of $\Lambda'$ across $w=0$,
and $\widetilde\Lambda_{\RR}=\widetilde\Lambda\otimes_{\ZZ}\RR$.
The translation $v/n$ is viewed in
$\widetilde\Lambda_{\RR}/\widetilde\Lambda$.
To see this normalization, lift the translation term to a holomorphic
section of the Lie bundle on the disk. Composing the generator
$n$ times gives a constant lattice translation $v$, invariant
under $T$. Subtracting the flat section $v/n$ leaves a translation
cocycle whose $n$ translates sum to zero. Averaging over the
$n$ group elements makes it a coboundary and removes it by a
holomorphic translation.

\begin{proposition}\label{finitequot}
In case~\textup{(ii)}, let $\sigma_s$ be a holomorphic family of
local period sections with $\sigma_s(0)=\sigma_0(0)$.
For $n=3$ or $4$, after shrinking the base and parameter disks,
the corresponding torus families over $\Delta_t^*$ admit a
proper flat filling with smooth total space and submersive
projection to $\Delta_s$. Each map $f_s:X_s\to\Delta_t$
has central fiber of multiplicity $n$ with the original reduced
surface. The fillings are diffeomorphic over the base, compatibly
with the smooth boundary attachments. Moreover, the natural
Hodge quotient identification over $\Delta_t^*$ extends to
\begin{align}\label{finiteimage}
 R^1f_{s*}\OO_{X_s}\simeq\cH/\cF_s=\sigma_s^*\mathcal{Q}
 \qquad\hbox{on }\Delta_t.
\end{align}
\end{proposition}
\begin{proof}
We give the details at the order-three point, where $t=w^3$.
The frame \eqref{finiteframe} is represented on the covering disk
by $c_1,c_2,wc_3,w^2c_4$. Put
\begin{align*}
 D(w)=\diag(1,1,w,w^2).
\end{align*}
We check that the perturbed Hodge planes downstairs lift
holomorphically across the center of this cover.

Because $\sigma_s-\sigma_0=O(t)$, there is a local holomorphic
automorphism $\Psi_s(t)=I+O(t)$ of $\cH$ taking $\cF_0$ to $\cF_s$.
Upstairs its expression is
\begin{align}\label{lift}
 \Psi_s^\sharp(w)=D(w)\Psi_s(w^3)D(w)^{-1}.
\end{align}
The entries of $\Psi_s(w^3)-I$ vanish to order at least three, whereas
$D(w)^{-1}$ has poles of order at most two. Consequently
$\Psi_s^\sharp(w)=I+O(w)$ extends holomorphically across $w=0$.
Since $D(\zeta w)=\rho D(w)$, this lift is equivariant.
Thus $\widetilde{\cF}_s=\Psi_s^\sharp\widetilde{\cF}_0$ is a Hodge
subbundle on the whole covering disk with exactly the original
central Hodge plane. On a sufficiently small closed disk, for
sufficiently small $s$, it satisfies \eqref{D} uniformly and
produces a smooth family of complex tori.

A flat real vector projects to a holomorphic section of
$(\widetilde{\cF}_s)^*$. Thus \eqref{affine} defines a holomorphic
deck action on every new family, with $n=3$. Its third power is the lattice
translation $v$, hence the identity on the torus. At $w=0$, both
the complex torus and the affine action are unchanged, so freeness
is preserved. At $w\ne0$ freeness follows from the base action.

At the order-four point, use the frame \eqref{finitefour} and put
$D_4(w)=\diag(1,1,w,w^3)$. The corresponding lift is
\begin{align}\label{liftfour}
 \Psi_s^\sharp(w)=D_4(w)\Psi_s(w^4)D_4(w)^{-1}.
\end{align}
Here $\Psi_s(w^4)-I$ vanishes to order at least four and $D_4^{-1}$
has poles of order at most three, so again $\Psi_s^\sharp=I+O(w)$.
The same argument preserves the affine action \eqref{affine}
with $n=4$.

Taking the free finite quotient supplies the new local filling,
holomorphically in the parameter $s$. It is proper because the
covering torus family and $\pi$ are proper. The quotient has smooth
total space, and its projection to $\Delta_s$ is a submersion.
Locally the map to $\Delta_t\times\Delta_s$ is $(w,y_1,y_2,s)
\mapsto(w^n,s)$, since the finite group acts freely. In particular,
this map is flat. Its underlying smooth model
is the same quotient of $\Delta_w\times(\widetilde\Lambda_{\RR}/\widetilde\Lambda)$
by the same affine action. In particular, the multiplicity, either
three or four, and the central reduced surface are retained.

As in the hexagonal case, these smooth identifications can be
chosen to match the smoothly varying boundary-collar
identifications used in Section~\ref{glue}, starting with the
original attachment.

To identify the direct image, fix $s$ and let
$g_s:\widetilde X_s\to\Delta_w$ denote the resulting smooth
torus family. The equivariant lifts \eqref{lift} and
\eqref{liftfour} give
\begin{align*}
 (\pi_*\widetilde{\cF}_s)^G
 =\Psi_s(\pi_*\widetilde{\cF}_0)^G=\cF_s.
\end{align*}
Finite pushforward is exact, and taking $G$-invariants is exact
by averaging. Applying these operations to the Hodge quotient
identity for the covering family gives
\begin{align*}
 R^1f_{s*}\OO_{X_s}
 &\simeq \bigl(\pi_*R^1g_{s*}\OO_{\widetilde X_s}\bigr)^G\\
 &\simeq \bigl(\pi_*(\widetilde{\cH}/\widetilde{\cF}_s)\bigr)^G
 \simeq\cH/\cF_s.
\end{align*}
Here the first isomorphism uses the free finite quotient
$\widetilde X_s\to X_s$, and the last uses \eqref{invext}
and the preceding identity. On the punctured disk this is the
natural Hodge quotient identification. The calculation holds
also at the branch point and includes the scheme-theoretic
multiple fiber.

\end{proof}

\subsection{Compatibility of the direct-image identifications}
\label{compare}

The local propositions identify the extensions over the entire
disks around the critical values, not just their restrictions to the punctured
disks. This is what determines the global quotient bundle and
its degree.

\begin{proposition}\label{localquotient}
Suppose the local fillings of Propositions~\ref{hexfill}
and~\ref{finitequot} are glued to the torus family over $B'$
by holomorphic identifications inducing the prescribed
identifications of the cohomology local system.
For the resulting map $f_s:X_s\to B$, the natural Hodge quotient
identification over $B'$ extends uniquely to an isomorphism
\begin{align*}
 R^1f_{s*}\OO_{X_s}\simeq\cH/\cF_s=\sigma_s^*\mathcal{Q}.
\end{align*}
\end{proposition}
\begin{proof}
Over $B'$ this is the Hodge quotient identification
\eqref{hodgequotient}. On each disk around a critical value, use
\eqref{heximage}
or \eqref{finiteimage}. These maps agree
with the natural identification on the punctured disk, so they
glue to the claimed isomorphism.
Two such extensions agree on the dense open set $B'$ and hence
everywhere.
\end{proof}

%%%%%%%%%%%%%%%%%%%%%%%%%%%%%%%%%%%%%%%%%%%%%%%%%%%%
\section{Global extension}\label{glue}

We first describe the smooth torus families over $B'=B\setminus\Sigma$
and explain why the original family admits a holomorphic section there.
Recall that $\Lambda'=\operatorname{Hom}(\mathcal{Z}',\ZZ_{B'})$
is the homology local system from Subsection~\ref{ss:period}.
We use $\cF'$ below for an arbitrary admissible Hodge subbundle
of $\cH'$. This includes the original $\cF'_0$ and the bundles
$\cF'_s=\cF_s|_{B'}$ of the smoothed period sections. The cohomology
bundle $\cH'$ and homology local system $\Lambda'$ remain fixed
as $s$ varies.
For such a $\cF'$, evaluation gives
\begin{align*}
 \Lambda'_{\RR,b}\longrightarrow\cF'(b)^*,\qquad
 \lambda\longmapsto\bigl(\omega\mapsto\omega(\lambda)\bigr).
\end{align*}
This real-linear map is injective: if a real class annihilates
$\cF'(b)$, it also annihilates its conjugate, and hence all of
$\cH'(b)$ by \eqref{D}. Both sides have real dimension four, so
evaluation is an isomorphism. Thus $\Lambda'_b$ maps to a lattice,
and the associated family of complex tori is
\begin{align}\label{Jac}
 \mathscr T'_{\cF'}=(\cF')^*/\Lambda'\longrightarrow B'.
\end{align}
As with $\cH'$, we use $\Lambda'_{\RR}$ also for its associated
flat real vector bundle on $B'$, with fiber
$\Lambda'_b\otimes_{\ZZ}\RR$. Taking the quotient in each fiber gives
a smooth bundle of real four-tori
\begin{align}\label{fixedreal}
 \mathscr M'=\Lambda'_{\RR}/\Lambda'\longrightarrow B',\qquad
 \mathscr M'_b=(\Lambda'_b\otimes_{\ZZ}\RR)/\Lambda'_b
 \cong\RR^4/\ZZ^4.
\end{align}
Its transition maps are induced by the integral monodromy matrices.

\begin{remark}
The two formulas \eqref{Jac} and \eqref{fixedreal} describe the
complex and smooth aspects of the same torus family. The bundle
$\mathscr M'$ uses only the integral homology local system and
is fixed throughout the construction. Evaluation identifies $\Lambda'_{\RR}$ with the underlying
real vector bundle of $(\cF')^*$, so a choice of Hodge planes equips $\mathscr M'$ with the complex structure
in \eqref{Jac}. For the original Hodge bundle $\cF'_0$, this
complex torus family is the relative Albanese of $f'_0:X'_0\to B'$:
its fiber over $b$ is
$H^0(X_{0,b},\Omega^1_{X_{0,b}})^*/H_1(X_{0,b},\ZZ)$.
The holomorphic vector bundle
$(\cF')^*$ is the tangent bundle along the zero section in the
fiber directions; the first coherent direct image is
$\cH'/\cF'$, by \eqref{hodgequotient}.
\end{remark}

The original torus fibers have no chosen origins. The following
lemma provides a holomorphic section over $B'$ and hence identifies
the original family with its relative Albanese.

\begin{lemma}\label{sections}
For every admissible $\cF'$, the torus family
$\mathscr T'_{\cF'}\to B'$ in \eqref{Jac} has the holomorphic
zero section $b\mapsto[0]\in\cF'(b)^*/\Lambda'_b$.
The original map $f'_0:X'_0\to B'$ also admits a holomorphic
section. Choosing such a section as origin gives an isomorphism
\begin{align*}
 X'_0\simeq\mathscr T'_{\cF'_0}
\end{align*}
over $B'$.
\end{lemma}
\begin{proof}
The zero section of $\mathscr T'_{\cF'}$ is the image of the
zero section of $(\cF')^*$ under the quotient map \eqref{Jac}.
To obtain a section of $f'_0$, let $\cT'$ be the sheaf of
holomorphic sections of $\mathscr T'_{\cF'}\to B'$, and let
$\cL'$ be the sheaf of holomorphic sections of $(\cF')^*$.
For $\cF'=\cF'_0$ these sheaves will be denoted by $\cT'_0$
and $\cL'_0$; for $\cF'=\cF'_s$ they will be denoted by
$\cT'_s$ and $\cL'_s$. For any of these choices, the quotient
map gives an exact sequence of sheaves of abelian groups on $B'$:
\begin{align}\label{toruseq}
 0\longrightarrow\Lambda'\longrightarrow\cL'
 \longrightarrow\cT'\longrightarrow0.
\end{align}
Surjectivity here is local: a holomorphic section of the torus
family lifts locally to its vector bundle. Since $B'$ is Stein
and has the homotopy type of a graph, Cartan's theorem B and
the cohomology sequence give
\begin{align*}
 H^1(B',\cL')=0,\qquad H^2(B',\Lambda')=0,
 \qquad H^1(B',\cT')=0.
\end{align*}

Apply this to $\cF'_0$. By the holomorphic submersion theorem,
$f'_0:X'_0\to B'$ has holomorphic sections locally. Choosing these as origins
identifies the local families with $\mathscr T'_{\cF'_0}$.
On overlaps, two choices differ by a holomorphic translation,
and these translations form a \v{C}ech cocycle with values in
$\cT'_0$. The vanishing of $H^1(B',\cT'_0)$ allows us to
adjust the local origins so that they agree. This gives the
required section and isomorphism.
\end{proof}

We now use these choices of origin to attach the local fillings
and compare the resulting fibration with the original one.

\begin{proposition}\label{topology}
For every sufficiently small $s\ne0$ in the comb smoothing,
there is a proper holomorphic fibration $f_s:X_s\to B$ with
connected fibers and a diffeomorphism
\begin{align*}
 \Phi_s:X_0\longrightarrow X_s,\qquad f_s\circ\Phi_s=f_0.
\end{align*}
The integral monodromy and the isotopy classes of the affine
attachments are unchanged, as are the critical values,
multiplicities, and reduced critical fibers. In particular,
\begin{align*}
 \pi_1(X_s)\simeq\pi_1(X_0),\qquad
 H_k(X_s;\ZZ)\simeq H_k(X_0;\ZZ)\quad\hbox{for every }k.
\end{align*}
The integral Leray spectral sequences of $f_s$ and $f_0$, including
their cup products, torsion, and differentials, are identified.
Thus $X_s$ is diffeomorphic to $S^6$ whenever $X_0$ is.
\end{proposition}
\begin{proof}
Fix a section of $f'_0$ as in Lemma~\ref{sections}, and write
$\mathscr T'_0=\mathscr T'_{\cF'_0}$. For each smoothed period
section, write $\mathscr T'_s=\mathscr T'_{\cF'_s}$. Its
distinguished section is the zero section $b\mapsto[0]$ of
Lemma~\ref{sections}.

Choose small disjoint disks around the critical values and annular
overlaps $A_i$ with $B'$, away from the comb teeth. Over
$A_i\times\Delta_s$, the new period sections and local fillings
extend holomorphically through $s=0$. On this product, let
$\cT'$ and $\cL'$ denote the section sheaves of the entire
family and its vector bundle; at a fixed parameter $s$ they are
the sheaves $\cT'_s$ and $\cL'_s$ just defined. We also write
$\Lambda'$ for the pulled-back homology local system.
The product is Stein and retracts onto a circle, so the proof of
Lemma~\ref{sections} again gives a holomorphic choice of origin in each local
filling over this product. It therefore identifies the filling with
the associated torus family on the overlap.

We must choose these identifications to agree with the original
attachments at $s=0$. Let $a_0$ be the translation over
$A_i\times\{0\}$ needed to obtain that agreement. The connecting
homomorphism in \eqref{toruseq} is surjective:
\begin{align*}
 \delta:H^0(A_i\times\Delta_s,\cT')\longrightarrow
 H^1(A_i\times\Delta_s,\Lambda').
\end{align*}
Restriction identifies the group on the right with
$H^1(A_i,\Lambda')$. Choose a translation $a$ on the product
whose image under $\delta$ restricts to $\delta(a_0)$. Then
$a_0-a|_{s=0}$ lifts to a section $\ell_0$ of $\cL'|_{s=0}$.
Cartan's theorem B, applied to $\mathcal I\cL'$ with
$\mathcal I$ the ideal of $s=0$, extends $\ell_0$ to a section
$\ell$ on the product. Adding the image of $\ell$ to $a$ gives
a holomorphic family of translations restricting exactly to $a_0$.
Using this translation makes the annular identifications agree
with the original ones at $s=0$.

Now glue $\mathscr T'_s$ to the perturbed local fillings by these
holomorphic identifications. One may use nested disks and annuli,
with gluing performed on collars; this gives a Hausdorff complex
manifold. Compactness follows by covering $B$ by a compact part
of $B'$ and smaller closed disks around the critical values, over
which the pieces are proper. The resulting map $f_s:X_s\to B$
is holomorphic, proper, and has connected fibers.

To construct the diffeomorphism, let $B^\circ$
be the complement of the interiors of smaller disks around the
critical values. By \eqref{fixedreal} and the chosen section of
$f'_0$, the parts of the original and new fibrations lying over
$B^\circ$ are smoothly identified with the same compact manifold
with boundary
\begin{align*}
 W=\mathscr M'|_{B^\circ}.
\end{align*}
Let $V_{i,s}$ denote the local filled closed disks. The constructions
above give diffeomorphisms
\begin{align*}
 a_{i,s}:V_{i,0}\longrightarrow V_{i,s},\qquad a_{i,0}=\id,
\end{align*}
preserving the maps to the base. Such trivializations also follow
directly by integrating lifts of parameter vector fields that
annihilate $t$; compactness ensures the required flows exist.
They can be prescribed on boundary collars, where the fibers are smooth.

After these identifications, write the boundary attachments as
\begin{align*}
 \varphi_{i,s}:\partial V_{i,0}\longrightarrow\partial_i W.
\end{align*}
They are a smooth family of diffeomorphisms, including at $s=0$,
and $\varphi_{i,0}$ is precisely the original attachment.
For a short path from $0$ to $s$, the maps
$\varphi_{i,s}\varphi_{i,0}^{-1}$ form an isotopy of $\partial_i W$
starting at the identity. Extend it to a diffeomorphism $K_s$ of
$W$, supported in disjoint boundary collars. This extension can
preserve the base: transport the boundary isotopy along the radial
coordinate of each collar and cut off its isotopy parameter there.
Now $K_s$ on $W$ and $a_{i,s}$ on each $V_{i,0}$ fit together,
because on a boundary point $p$ they satisfy
\begin{align*}
 K_s(\varphi_{i,0}(p))=\varphi_{i,s}(p).
\end{align*}
They define a diffeomorphism $\Phi_s:X_0\to X_s$ over $B$.

The fixed homology local system identifies the integral monodromy.
The local filling propositions preserve the critical values,
multiplicities, and reduced critical fibers, and the boundary
isotopy above preserves the isotopy classes of the affine
attachments. The claims about fundamental groups and integral
homology follow from $\Phi_s$. Since $f_s\circ\Phi_s=f_0$, this
diffeomorphism also identifies the derived direct images of the
constant sheaf, and hence their Leray spectral sequences with
products. This proves all the claims.
\end{proof}

\begin{remark}
There is no global holomorphic family of compact manifolds joining
$X_s$ to $X_0$ in this argument. At $s=0$, the global period
object is the comb, with entire conics over the $b_i$, rather than
a section. The identification with $X_0$ is a smooth one; it does
not imply constancy of the coherent direct-image degree.
\end{remark}

For each sufficiently small $s\ne0$, Proposition~\ref{localquotient}
now gives the global identification
\begin{align}\label{EisQ}
 E_s:=R^1f_{s*}\OO_{X_s}\simeq\sigma_s^*\mathcal{Q}=\cH/\cF_s .
\end{align}
This also holds for the original fibration when $s = 0$, where $\sigma_0$ is the original period section (not the comb). 
The calculation in Subsection~\ref{conics} now proves the degree change formula
\begin{align} 
\label{degreechange}
 \deg E_s =\deg E_0+2m,
\end{align}
which will be used below in Section~\ref{applications}.

%%%%%%%%%%%%%%%%%%%%%%%%%%%%%%%%%%%%%%%%%%%5
\section{Deformations}\label{s:deformations}

Fix a sufficiently small smooth section $\sigma_m$ obtained from the
positive comb in Proposition~\ref{comb}. In this section the base point
of the deformation is this smoothed section, not the nodal comb.

\begin{lemma}\label{normalbundle}
Defining $\cF_m=\sigma_m^*\mathcal{S}$ and $E_m=\sigma_m^*\mathcal{Q}$, the normal
bundle of $\sigma_m(B)$ satisfies
\begin{align}\label{parambundles}
 N_m:=N_{\sigma_m(B)/Y}\simeq\cF_m^*\otimes E_m.
\end{align}
The same identification holds for every holomorphic section of
$Y\to B$, including the original section $\sigma_0$.
\end{lemma}
\begin{proof}
Let $T_{Y/B}=\ker(dp)$ be the tangent bundle along the fibers
of the Grassmannian bundle $p:Y=\operatorname{Gr}_B(2,\cH)\to B$.
At a plane $W\subset\cH(b)$, nearby planes can be written as
graphs of linear maps from $W$ to a chosen complement. Thus the
tangent space to the Grassmannian at $W$ is naturally
$\Hom(W,\cH(b)/W)$. As $W$ varies, these identifications give
$T_{Y/B}\simeq\sHom(\mathcal{S},\mathcal{Q})$. For any holomorphic section
$\sigma:B\to Y$, the identity $p\circ\sigma=\id_B$ shows
that $d\sigma$ splits the tangent sequence along its image.
Consequently,
\begin{align*}
 N_{\sigma(B)/Y}\simeq\sigma^*T_{Y/B}
 \simeq(\sigma^*\mathcal{S})^*\otimes\sigma^*\mathcal{Q}.
\end{align*}
Taking $\sigma=\sigma_m$ proves \eqref{parambundles}.
\end{proof}

As before, we regard $\Sigma$ as a reduced divisor. In this section
it is the set of critical values, with no additional regular values.

\begin{lemma}\label{paramspace}
The germ $T_m$ of sections of $Y\to B$ agreeing with $\sigma_m$ at
$\Sigma$ is smooth at $\sigma_m$, with tangent space
\begin{align}\label{paramtangent}
 V_m=H^0(B,N_m(-\Sigma)).
\end{align}
Writing $c_m=\deg E_m$, $h=\deg\cH$, and $r=|\Sigma|$, its
dimension is
\begin{align}\label{paramdim}
 d_m=4+4c_m-2h-4r.
\end{align}
After shrinking to a polydisc centered at $\sigma_m$, these sections
give a proper holomorphic submersion $\pi:\mathscr X\to T_m$ with
central fiber $X_m$, and a holomorphic map $F:\mathscr X\to B$.
Its restrictions $f_u:X_u\to B$ are the fibrations of the
construction, with critical values $\Sigma$.
\end{lemma}

\begin{proof}
Infinitesimal deformations of a section with its values fixed at $\Sigma$
are sections of $N_m(-\Sigma)$, and their obstructions lie in
$H^1(B,N_m(-\Sigma))$. The latter group vanishes by
\eqref{normalvan}. Thus the section germ is smooth with the stated
tangent space. This is a space of sections, not of parametrized
curves, so no reparametrizations of the handle are counted.
The exact sequence
\begin{align*}
 0\longrightarrow\cF_m\longrightarrow\cH
 \longrightarrow E_m\longrightarrow0
\end{align*}
gives $\deg\cF_m=h-c_m$. Hence
\begin{align*}
 \rank N_m=4,\qquad
 \deg N_m=2c_m-2(h-c_m)=4c_m-2h.
\end{align*}
Riemann--Roch on $B=\PP^1$, together with the vanishing above, gives
$h^0(N_m(-\Sigma))=4+\deg N_m-4r$, proving \eqref{paramdim}.

The local fillings of Propositions~\ref{hexfill} and~\ref{finitequot}
work with a parameter polydisc in place of a disk. Indeed, the
period matrices depend holomorphically on all parameters and are
uniformly bounded after shrinking. The values at $\Sigma$ are fixed,
so the lifts \eqref{lift} and \eqref{liftfour} still extend over the
covering disks. The admissibility and properness estimates in those
proofs are uniform on a smaller polydisc. On the compact complement
of the critical disks, admissibility follows from openness. The local
equations $t=x_1x_2x_3$ and $t=w^n$ are independent of the
parameters, so the filled spaces project submersively to $T_m$.

The gluing proof of Proposition~\ref{topology} also applies with
$T_m$ in place of $\Delta_s$: each annular product $A_i\times T_m$
is Stein and retracts onto a circle. The argument with
\eqref{toruseq} therefore gives holomorphic annular identifications
extending the chosen ones for $X_m$ at $u=0$. Gluing on nested
collars gives a Hausdorff holomorphic family, and the finite cover
of $B$ by a compact part of $B'$ and closed critical disks proves
properness over a smaller parameter polydisc. The projection is a
submersion both over $B'$ and in the filled charts.
\end{proof}

\begin{proposition}\label{effective}
If $r\geq3$, the family in Lemma~\ref{paramspace} gives
\begin{align}\label{mainmoduli}
 h^1(X_m,\Theta_{X_m})\geq
 \dim\operatorname{Def}(X_m)\geq d_m.
\end{align}
\end{proposition}

\begin{proof}
We first show that only countably many parameters $u\in T_m$
can give a manifold biholomorphic to $X_m$. Every $X_u$ is
diffeomorphic to $X_m$, hence to $S^6$, by Ehresmann's theorem.
The proof of Proposition~\ref{analytic} uses only
$H^2(X_u,\ZZ)=0$, connected fibers, and a general fiber that is a
complex torus. It therefore applies to every $f_u$, showing that
$f_u$ is the algebraic reduction.

Consequently, a biholomorphism $X_u\simeq X_m$ induces an
automorphism $\alpha$ of $B$ preserving $\Sigma$. There are only
finitely many such $\alpha$, since $r\geq3$. It also induces an
isomorphism of the integral cohomology local systems over $B'$,
covering $\alpha$. Under the fixed markings of the construction,
there are only countably many such isomorphisms: each is determined
by its value on one fiber, an integral invertible matrix. For each
pair consisting of $\alpha$ and a local-system isomorphism, preservation
of the Hodge planes determines the period section $\sigma_u$ on
$B'$, and hence on $B$. Since $T_m$ parametrizes sections, this
determines at most one $u$. This proves the countability claim.

The holomorphic family supplies a classifying map of germs
\begin{align*}
 \kappa:(T_m,0)\longrightarrow
 (\operatorname{Def}(X_m),0).
\end{align*}
Its fiber over $0$ is an analytic subset consisting of parameters
for manifolds biholomorphic to $X_m$. It is countable by the
preceding argument, so it has dimension zero. The dimension
inequality for holomorphic maps gives
$\dim\operatorname{Def}(X_m)\geq\dim T_m=d_m$.
Finally, the Kuranishi germ is an analytic subspace of
$H^1(X_m,\Theta_{X_m})$, proving the other inequality
in \eqref{mainmoduli}.
\end{proof}

%%%%%%%%%%%%%%%%%%%%%%%%%%%%%%%%%%%%%%%%%
\section{Applications}\label{applications}
We now apply the construction to the Alp\"oge-Claude configuration 
\cite{AC,Engel}. We first compute the bundles needed for the
comb-smoothing bound and obtain the sufficient bound $m\geq12$.
We then establish the Hodge numbers and deformation estimate in
Theorem~\ref{t:main}. We also describe the curves and divisors,
identify the algebraic reduction, and prove that all plurigenera vanish
before proving the blow-up corollary for $\CP^3$.

We recall the construction from 
\cite[Sections~2--3]{Engel}.
Let $\pi:S\to B=\PP^1$ be the rational elliptic surface with
singular fiber configuration $\mathrm{IV}^*\,\mathrm{III}\,\mathrm{I}_1$, 
together with zero section $O$ and a generator $P$ of its Mordell--Weil group.
The complement $M^\times$ of the zero section in
$M=\OO_S(P-O)$ is a principal $\CC^*$-bundle over $S$.
On a smooth elliptic fiber $S_b$, the restriction of $M$ has
degree zero. The corresponding surface $M^\times|_{S_b}$ has
universal cover $\CC^2$ and a rank-three period lattice.
A fiberwise linear lift of translation by $-6P$ 
gives a cyclic action
on $M^\times|_{\pi^{-1}(B')}$. After multiplying the lift by a
sufficiently small nonzero scalar, its quotient is the smooth
two-torus family $f'_0:X'_0\to B'$: the cyclic quotient adds
the fourth period. The local fillings then extend this family
over $\Sigma$.

The positively oriented circle in each $\CC^*$ fiber gives a
nonzero homology class $\delta_b\in H_1(X_{0,b},\ZZ)$.
Multiplication by a nonzero complex number preserves this circle
class, and the lifted translations commute with the principal
$\CC^*$ action. Thus these classes are preserved by parallel
transport and define a single section $\delta$ of the homology
local system $\Lambda'$ over $B'$. This is the circle class used
below. Holomorphic differentials on $S_b$ pull back to
$M^\times|_{S_b}$ and descend through the cyclic quotient,
because translation on an elliptic curve preserves its
holomorphic differentials. Their integrals around $\delta_b$
are zero.

\subsection{The initial bundles}\label{boundexplicit}

Integration around the circle class defines the holomorphic map
\begin{align*}
 \cF'_0\longrightarrow\OO_{B'},\qquad
 \omega\longmapsto\int_{\delta}\omega.
\end{align*}
It is nonzero on every fiber over $B'$. Indeed, if all holomorphic
one-forms integrated to zero on the real class $\delta_b$, their
conjugates would also do so. Since these two spaces span
$H^1(X_{0,b},\CC)$, this would imply $\delta_b=0$.
The map is therefore surjective, and its kernel is precisely the
line of elliptic differentials described above.

We next check its extension across the critical values. Near a
multiple fiber, the branched cover $t=w^n$ gives a smooth torus
family, as in case~(ii) of Proposition~\ref{periodextension}.
The class $\delta$ extends as a nonzero homology class on this
cover and is fixed by the cyclic action. The preceding argument
therefore gives a local holomorphic relative one-form whose integral on
$\delta$ is one. Averaging this form over the cyclic group gives
an invariant form with the same integral, since the action fixes
$\delta$. Consequently integration defines a surjection to
$\OO_{\Delta_t}$ on the invariant extension $\cF_0$.

At the hexagonal fiber, the monodromy fixes exactly the span of
the angular classes $A_1,A_2$ in Subsection~\ref{localmodels}.
Thus $\delta=a_1A_1+a_2A_2$ for constants $a_1,a_2$ not both
zero. On the punctured disk, the two Hodge generators in
\eqref{extendedhodgeframe} have constant integrals $a_1,a_2$
on $\delta$. These values define the extended evaluation map
at the central value and prove surjectivity there
as well. 
We obtain a surjective map
$\cF_0\to\OO_{\PP^1}$; computing the extension of its kernel
will determine~$\cF_0$.

The usual Hodge line of the rational elliptic surface is
$\pi_*\omega_{S/B}=\OO_{\PP^1}(1)$; see
\cite[Theorems~12.1 and~12.3]{KodairaIII}. The extension of its
differentials as sections of $\cF_0$ differs at the two critical
values with multiple fibers, corresponding to $\mathrm{IV}^*$ and $\mathrm{III}$ on $S$.
At $t=0$, take $t=w^3,\ x=w^4X,\ y=w^6Y$. The good-reduction
equation and differential are
\begin{align*}
 Y^2=X^3-3w(w^3-1)X+2(w^3-1)^2,\qquad
 \frac{dx}{y}=w^{-2}\frac{dX}{Y}.
\end{align*}
The deck generator $\gamma$, covering $w\mapsto\zeta_3w$,
satisfies $\gamma^*(dX/Y)=\zeta_3^2dX/Y$. Multiplication
by $w$ therefore gives the invariant generator
$w\,dX/Y=t\,dx/y$.
At $t=1$, take $t-1=w^4,\ x=w^2X,\ y=w^3Y$, obtaining
\begin{align*}
 Y^2=X^3-3(1+w^4)^3X+2w^2(1+w^4)^4,\qquad
 \frac{dx}{y}=w^{-1}\frac{dX}{Y}.
\end{align*}
Here $\gamma^*(dX/Y)=i\,dX/Y$, so the invariant generator
is $w^3dX/Y=(t-1)dx/y$.
At the nodal elliptic fiber underlying the hexagonal model, the
elliptic differential has its usual nonzero logarithmic extension,
so no additional modification occurs.
Consequently the kernel line is
$\OO_{\PP^1}(1)(-[0]-[1])=\OO_{\PP^1}(-1)$, and
\begin{align*}
 0\longrightarrow\OO_{\PP^1}(-1)\longrightarrow\cF_0
 \longrightarrow\OO_{\PP^1}\longrightarrow0
\end{align*}
splits because $H^1(\OO_{\PP^1}(-1))=0$.

We use the flat connection on $\cH'$ to compute the degree of
$\cH$. In the frame \eqref{logframe} at the hexagonal point,
\begin{align*}
 \nabla e_i=f_i\frac{dt}{2\pi i t},\qquad \nabla f_i=0,
\end{align*}
so the coefficient of $dt/t$ has trace zero. At the order-three
and order-four points, the explicit connections in the proof of
Proposition~\ref{periodextension} have diagonal coefficients
\begin{align*}
 (0,0,1/3,2/3),\qquad (0,0,1/4,3/4),
\end{align*}
respectively.
If $R_p$ denotes the coefficient of $dt/t$ at $t=0$
in a local frame near $p$, the degree formula is
\begin{align*}
 \deg\cH=-\sum_{p\in\Sigma}\tr R_p.
\end{align*}
Indeed, the induced connection on $\det\cH$ has coefficient
$\tr R_p$. In a meromorphic section of this line bundle, a zero
or pole of order $k$ adds $k\,dt/t$ to its connection form.
The residue theorem therefore gives the displayed formula.
Here the hexagonal point contributes zero and each of the other
two critical values contributes one. Thus $\deg\cH=-2$, and
$\deg E_0=\deg\cH-\deg\cF_0=-1$.
The Leray isomorphism \eqref{transgression} and
Lemma~\ref{nowhere} supply a
nowhere-vanishing section of $E_0$. It follows that
\begin{align*}
 E_0=\OO_{\PP^1}\oplus\OO_{\PP^1}(-1),\qquad
 \cF_0=\OO_{\PP^1}\oplus\OO_{\PP^1}(-1).
\end{align*}

These calculations also prove the initial Hodge numbers 
of $X_0$. 
\begin{proposition} 
\label{ACHodge}
For the Alp\"oge-Claude complex structure, we have
\begin{align}\label{ACHodgenum}
 h^{0,1}(X_0)=1,\quad h^{0,2}(X_0)=0,\quad h^{0,3}(X_0)=0,
\quad  \deg(E_0) = -1.
\end{align}
\end{proposition}
\begin{proof}
Indeed, Theorem~\ref{allfree} gives
\begin{align*}
 R^2f_{0*}\cO_{X_0}\simeq\det E_0
 \simeq\OO_{\PP^1}(-1).
\end{align*}
Since $E_0=\OO_{\PP^1}\oplus\OO_{\PP^1}(-1)$ and
$H^1(\PP^1,E_0)=0$, the Leray isomorphisms
\eqref{e:lerayiso} and exact sequence \eqref{e:lerayseq} yield \eqref{ACHodgenum}.
The degree formula follows from this. 
\end{proof}
The degree formula \eqref{degreechange} therefore gives $\deg E_m=2m-1$
after smoothing a comb with $m$ teeth.

\subsection{The smoothing bound}\label{smoothingbound}
By Lemma~\ref{normalbundle}, applied to $\sigma_0$, the normal
bundle to the original section is
$\sigma_0^*T_{Y/B}\simeq\cF_0^{\,*}\otimes E_0$.
We have
\begin{align}\label{originalnormal}
 N_{\sigma_0(B)/Y}
 =\cF_0^{\,*}\otimes E_0
 =\OO_{\PP^1}(-1)\oplus\OO_{\PP^1}^{\oplus2}\oplus\OO_{\PP^1}(1).
\end{align}
Fixing its value at all three critical values means $r=3$.
Bound~\eqref{bound} becomes
\begin{align*}
 m\geq(3-(-1))+3+3+(3-1)=12.
\end{align*}
Hence this argument gives $c=2m-1$ for every $m\geq12$.

\subsection{Hodge numbers and deformations}\label{nonnegative}
The sufficient bound above ensures \eqref{normalvan}, so
Proposition~\ref{effective} applies to the families just constructed.
We have $h=-2$, $r=3$, and $c_m=2m-1$, giving
\begin{align*}
 d_m=4+4(2m-1)-2(-2)-4\cdot3=8m-8
 \qquad(m\geq12).
\end{align*}

These calculations prove \eqref{intromoduli}.
Each additional conic increases $d_m$ by eight. Here we are counting
nearby final period sections with the critical values and limiting periods fixed,
not the separate choices used to produce a comb.

Since $X_0$ is diffeomorphic to $S^6$, 
Proposition~\ref{topology} gives the
same conclusion for $X_m$. Proposition~\ref{degrees} therefore applies
to every $f_m$. Here $c_0=-1$, so $c_m=2m-1$; for $m>0$ this excludes
the negative splitting alternative. Thus $a_m,b_m\geq0$ and
\eqref{globalhodge} gives
\begin{align*}
 h^{0,1}(X_m)=c_m+2=2m+1,\qquad
 h^{0,2}(X_m)=c_m+1=2m,\qquad h^{0,3}(X_m)=0.
\end{align*}
These are the Hodge-number identities in Theorem~\ref{t:main}.
Evaluation of global sections
of $E_m$ onto a fiber is surjective, so the restriction map
$H^1(X_m,\cO_{X_m})\to H^1(F,\cO_F)$ has rank two.
Different $m$ give nonbiholomorphic manifolds, distinguished by
$h^{0,1}$. Together with the deformation estimate above, this proves
the Hodge-number and deformation claims of Theorem~\ref{t:main}.
The algebraic-reduction claim is proved below. Note that the argument
does not determine $a_m$ and $b_m$ individually.

For comparison, any complex structure on $S^6$ satisfies
$h^{0,3}=0$ and $h^{0,1}=h^{0,2}+1\geq1$; see
\cite[Subsections~2.1--2.2]{Angella}. The original example $X_0$
attains the lower bound, while the manifolds $X_m$ realize
arbitrarily large values of both $h^{0,1}$ and $h^{0,2}$.

\subsection{Some other analytic properties}\label{analyticprops}

We next state the algebraic dimension and the absence of curves
projecting onto the base, and then compute the Kodaira dimension.

\begin{proposition}\label{analytic}
Every compact irreducible curve and every irreducible divisor in
$X_m$ is contained in a fiber of $f_m$. In particular, $f_m$
has no holomorphic section or multisection over $B$.
Moreover, $a(X_m)=1$, and $f_m:X_m\to B$ is the algebraic reduction.
The same conclusions hold for every member of the family in
Lemma~\ref{paramspace}.
\end{proposition}
\begin{proof}
Let $\nu:\widetilde C\to X_m$ be the normalization map of a
compact irreducible curve. Since $H^2(X_m,\ZZ)=0$,
\begin{align*}
 \deg\nu^*f_m^*\OO_B(1)=0.
\end{align*}
A nonconstant map $f_m\circ\nu:\widetilde C\to\PP^1$
would have positive degree. Thus the curve is contained in a fiber.

An effective divisor projecting onto $B$ would restrict to a nonzero
effective divisor on a general torus fiber. Its first Chern class
would be zero, again because $H^2(X_m,\ZZ)=0$, contradicting
its positive intersection with a K\"ahler class on the torus.
Every divisor is therefore vertical.

Let $h$ be a meromorphic function on $X_m$, with pole divisor $P$.
Choose an effective divisor $D$ on $B$ with $f_m^*D\geq P$.
The projection formula and $f_{m*}\OO_{X_m}=\OO_B$ give
\begin{align*}
 h\in H^0(X_m,f_m^*\OO_B(D))=H^0(B,\OO_B(D)).
\end{align*}
Thus every meromorphic function is pulled back from $B$, and the
meromorphic function field of $X_m$ is $\CC(\PP^1)$.
This proves the algebraic-reduction claim and $a(X_m)=1$.
The argument applies to every nearby member of the family in
Lemma~\ref{paramspace}, since Ehresmann's theorem makes each
of them diffeomorphic to $S^6$ and their general fibers are still tori.
\end{proof}

Consequently, a biholomorphism between two of these manifolds
identifies their fibrations up to an automorphism of the base.
The argument of Lemma~\ref{sections} also gives a holomorphic section of
$f'_m:X'_m=f_m^{-1}(B')\to B'$, but Proposition~\ref{analytic}
shows that no such section extends over $B$.

\begin{proposition}\label{canonical}
All plurigenera of $X_m$ vanish, so $\kappa(X_m)=-\infty$.
\end{proposition}
\begin{proof}
The coherent sheaf $f_{m*}K_{X_m}$ is torsion-free on the smooth
curve $B$, and its rank is one because a smooth torus fiber has a
one-dimensional space of holomorphic two-forms. Thus
$f_{m*}K_{X_m}\simeq\OO_B(d)$ for some integer $d$.
By \eqref{cupiso}, $R^2f_{m*}\OO_{X_m}=\OO_B(2m-1)$.
For $k$ sufficiently large, the projection formula, ordinary
Serre duality on $X_m$, and the Leray spectral sequence give
\begin{align*}
 d+k+1
 &=h^0(B,f_{m*}K_{X_m}\otimes\OO_B(k))\\
 &=h^0(X_m,K_{X_m}\otimes f_m^*\OO_B(k))
 =h^3(X_m,f_m^*\OO_B(-k))\\
 &=h^1(B,\OO_B(2m-1-k))
 =k-2m.
\end{align*}
Consequently $d=-2m-1$.

The evaluation map $f_m^*(f_{m*}K_{X_m})\to K_{X_m}$ is an
isomorphism over the smooth fibers, since a nonzero holomorphic
two-form on a torus is nowhere vanishing. Its zero divisor $Z$ is
therefore effective and supported on the critical fibers, and
\begin{align*}
 K_{X_m}\simeq f_m^*(f_{m*}K_{X_m})\otimes\OO_{X_m}(Z).
\end{align*}
Every reduced fiber is irreducible: the hexagonal fiber has
normalization $dP_6$, and each reduced multiple fiber is a
quotient of a torus; see Section~\ref{fillings}.
Write $f_m^*(p)=nD$ at a critical value, and let $a$ be the
coefficient of $D$ in $Z$. Then $a<n$. Indeed, a local frame
of $f_{m*}K_{X_m}$ corresponds to a canonical form $\alpha$ on a
neighborhood of this fiber. If $a\geq n$, division by a base
coordinate $t$ vanishing at $p$ would leave $\alpha/t$ holomorphic,
since $X_m$ is smooth. This would express a frame of $f_{m*}K_{X_m}$
as $t$ times a holomorphic section of $f_{m*}K_{X_m}$, which is impossible.

It follows that $Z\leq f_m^*\Sigma$, with $\Sigma$ regarded as a
reduced divisor, and hence
\begin{align*}
 K_{X_m}\hookrightarrow
 f_m^*\bigl(f_{m*}K_{X_m}\otimes\OO_B(\Sigma)\bigr).
\end{align*}
The line bundle on the base has degree $-2m-1+3=-2m+2<0$.
Taking any positive tensor power and using $f_{m*}\OO_{X_m}=\OO_B$
therefore gives $H^0(X_m,K_{X_m}^{\otimes q})=0$ for every $q\geq1$.
\end{proof}

\subsection{Proof of Corollary~\ref{cpthree}}\label{blowups}
Let $X$ be a compact complex threefold diffeomorphic to $S^6$,
and choose $p\in X$. The underlying smooth manifold of its blow-up is
\begin{align*}
 \widehat X=\operatorname{Bl}_pX
 \cong S^6\mathbin{\#}\overline{\CP}^{\,3}
 \cong\overline{\CP}^{\,3}.
\end{align*}
Complex conjugation on $\CP^3$ reverses orientation, so
$\widehat X$ is diffeomorphic to the standard smooth $\CP^3$.

We recall the relation between deformations of a blow-up and
deformations of a manifold with a marked point. The exceptional
divisor satisfies
\begin{align*}
 E\cong\CP^2,\qquad N_{E/\widehat X}\cong\OO_{\CP^2}(-1),
 \qquad H^0(E,N_{E/\widehat X})=H^1(E,N_{E/\widehat X})=0.
\end{align*}
Kodaira's stability theorem \cite{KodairaStability} implies that
$E$ persists uniquely in small deformations of $\widehat X$.
Since $\CP^2$ is rigid and its normal bundle remains
$\OO_{\CP^2}(-1)$, the relative blow-down
\cite[Proposition~2]{FujikiBlowdown} contracts these divisors
to a holomorphic section in a deformation of $X$. Conversely,
blowing up such a section gives a deformation of $\widehat X$.
These constructions are inverse and remain compatible with
infinitesimal deformations.

For any deformation of $X$, choosing the marked point locally
introduces three coordinates and no further obstructions. Infinitesimal motions arising from
automorphisms are the values of global holomorphic vector fields.
Put
\begin{align*}
 r_p=\operatorname{rank}(\operatorname{ev}_p),\qquad
 \operatorname{ev}_p:H^0(X,\Theta_X)\longrightarrow T_pX,
 \quad V\longmapsto V(p).
\end{align*}
Thus the additional tangent directions form
$T_pX/\operatorname{im}(\operatorname{ev}_p)$, of dimension $3-r_p$.
The standard comparison of semiuniversal deformation spaces
\cite[Exercise~15.7]{HartshorneDef} converts this unobstructed
choice of point into the local dimension formula
\begin{align*}
 \dim\operatorname{Def}(\widehat X)
 =\dim\operatorname{Def}(X)+3-r_p.
\end{align*}

For $X=X_m$, Proposition~\ref{analytic} implies that every
automorphism preserves $f_m$ up to an automorphism of $B$.
The connected group $\operatorname{Aut}^0(X_m)$ fixes each
critical value, since the set of critical values is finite.
There are three such values,
so its induced action on $\PP^1$ is trivial. For $p\in X_m$, the orbit
map $a_p:\operatorname{Aut}^0(X_m)\to X_m$, $g\mapsto g(p)$, has
differential $\operatorname{ev}_p$ at the identity. Left translation in
the group corresponds to composition with an automorphism of $X_m$,
so this map has constant rank $r_p$. By the holomorphic constant-rank
theorem, its local images have complex dimension $r_p$. They lie in
the fiber $f_m^{-1}(f_m(p))$, which has dimension two even when it
is singular. Hence $r_p\leq2$. Consequently
\begin{align*}
 \dim\operatorname{Def}(\widehat X_m)
 \geq\dim\operatorname{Def}(X_m)+1\geq8m-7
 \qquad(m\geq12),
\end{align*}
which proves \eqref{cpmoduli}.

Finally, blowing up a point preserves the numbers $h^{0,q}$,
without a K\"ahler hypothesis; see \cite[Corollary~1.4]{RYY}.
Thus the claimed Hodge numbers follow from \eqref{newhodge}.
Since 
\begin{align*}b^1(\PP^3) = 0 \neq 2 h^{0,1}(\widehat{X}_m) = 2(2m +1),
\end{align*} 
$\widehat{X}_m$ cannot be K\"ahler. 
Blowing up preserves the meromorphic function field, so
$a(\widehat X_m)=a(X_m)=1$ as well.
Birational invariance of Kodaira dimension implies that
$\kappa(\widehat X_m)=\kappa(X_m)=-\infty$.

\section{Questions}\label{future}

The construction leaves several questions.

\begin{enumerate}
\item \emph{The full deformation space.}
What are $h^1(X_m,\Theta_{X_m})$ and the exact local Kuranishi
dimension? The construction gives a lower bound for the local
Kuranishi dimension, but does not determine the full Kuranishi germ.
Does every sufficiently small deformation of $X_m$ retain a
holomorphic torus fibration over $\PP^1$?

\item \emph{The remaining Hodge numbers.}
What is the full Hodge diamond of $X_m$? Beyond the identities
established here and their consequences under Serre duality, which
Hodge numbers are determined by $m$, and which can vary within
the constructed families? The corresponding Hodge numbers of
$\widehat X_m$ then follow from the blow-up formula
\cite[Theorem~1.1]{RYY}.

\item \emph{Relations between the examples.}
How are the families for different values of $m$ related in the
moduli space? Can they be joined by chains of holomorphic
deformations, or do they belong to distinct deformation classes?

\item \emph{The splitting types and the first examples.}
The construction determines $a_m+b_m=2m-1$, but not the individual
degrees in
\begin{align*}
 R^1f_{m*}\cO_{X_m}
 =\cO_{\PP^1}(a_m)\oplus\cO_{\PP^1}(b_m).
\end{align*}
Which pairs $(a_m,b_m)$ occur? What is the smallest positive value of $m$
for which the construction can be carried out starting from $X_0$?
The bounds obtained here are sufficient and need not be optimal.

\item \emph{Further examples and degeneration models.}
Beyond the configuration treated here, can period modification
be applied to the other known examples?
The $2$-parameter family in \cite{Via}, 
based on the configuration $\mathrm{III}^*\,\mathrm{I}_1\,\mathrm{I}_1\,\mathrm{I}_1$, 
has two rank-one degenerations, which would require a further local argument in order to adapt the period modification. The Liu--Rollenske family~\cite{LR26}, based on the configuration $\mathrm{III}^*\,\mathrm{II}\,\mathrm{I}_1$, has a cuspidal
degeneration, which would also require further arguments.
The same question applies to Engel's additional examples in the
\emph{Threefold Explorer}~\cite{Explorer}, obtained from other
rational elliptic-surface data in the Oguiso--Shioda
classification~\cite{OS}.
More generally, which of the local hypotheses in
Subsection~\ref{localmodels} can be weakened, and what other
degeneration models admit the same construction?
\end{enumerate}

\bibliographystyle{amsplain}
\bibliography{Viaclovsky_Comb_S6_References}

\providecommand{\bysame}{\leavevmode\hbox to3em{\hrulefill}\thinspace}
\providecommand{\MR}{\relax\ifhmode\unskip\space\fi MR }
\providecommand{\MRhref}[2]{%
  \href{http://www.ams.org/mathscinet-getitem?mr=#1}{#2}
}
\providecommand{\href}[2]{#2}
\begin{thebibliography}{10}

\bibitem{ABGKR}
Ilka Agricola, Giovanni Bazzoni, Oliver Goertsches, Panagiotis Konstantis, and
  S\"{o}nke Rollenske, \emph{On the history of the {H}opf problem},
  Differential Geom. Appl. \textbf{57} (2018), 1--9.

\bibitem{Angella}
Daniele Angella, \emph{Hodge numbers of a hypothetical complex structure on
  {$S^6$}}, Differential Geom. Appl. \textbf{57} (2018), 105--120.

\bibitem{AHV}
Jos\'{e}~Manuel Aroca, Heisuke Hironaka, and Jos\'{e}~Luis Vicente,
  \emph{Complex analytic desingularization}, Springer, Tokyo, 2018.

\bibitem{Besse}
Arthur~L. Besse, \emph{Einstein manifolds}, Ergebnisse der Mathematik und ihrer
  Grenzgebiete (3), vol.~10, Springer-Verlag, Berlin, 1987.

\bibitem{BuskinIzadi}
Nikolay Buskin and Elham Izadi, \emph{Twistor lines in the period domain of
  complex tori}, Geom. Dedicata \textbf{213} (2021), 21--47.

\bibitem{CDP98}
Fr\'{e}d\'{e}ric Campana, Jean-Pierre Demailly, and Thomas Peternell, \emph{The
  algebraic dimension of compact complex threefolds with vanishing second
  {B}etti number}, Compositio Math. \textbf{112} (1998), no.~1, 77--91.

\bibitem{CDP20}
\bysame, \emph{The algebraic dimension of compact complex threefolds with
  vanishing second {B}etti number}, Compos. Math. \textbf{156} (2020), no.~4,
  679--696.

\bibitem{AC}
{Claude} and Levent Alp{\"o}ge, \emph{The $(3,4,\infty)$ modular family of
  $2$-tori, completed at its three special points, is a complex structure on
  {$S^6$}}, Preprint, \url{https://alpo.ge/s6.pdf}, 2026.

\bibitem{Engel}
Philip Engel, \emph{Complex structures on {$S^6$}}, Preprint,
  \url{https://philip-engel.github.io/S6.pdf}, 2026.

\bibitem{Explorer}
\bysame, \emph{{Threefold Explorer}},
  \url{https://philip-engel.github.io/threefold-explorer/}, 2026.

\bibitem{FujikiBlowdown}
Akira Fujiki, \emph{On the blowing down of analytic spaces}, Publ. Res. Inst.
  Math. Sci. \textbf{10} (1974/75), 473--507.

\bibitem{GabberLodh}
Ofer Gabber and R\'emi Lodh, \emph{Cohomological flatness over discrete
  valuation rings: numerical and logarithmic criteria},
  \href{https://arxiv.org/abs/2403.01831}{arXiv:2403.01831}, 2024.

\bibitem{GHS}
Tom Graber, Joe Harris, and Jason Starr, \emph{Families of rationally connected
  varieties}, J. Amer. Math. Soc. \textbf{16} (2003), no.~1, 57--67.

\bibitem{GH}
Phillip Griffiths and Joseph Harris, \emph{Principles of algebraic geometry},
  Pure and Applied Mathematics, Wiley-Interscience [John Wiley \& Sons], New
  York, 1978.

\bibitem{GrothendieckTohoku}
Alexander Grothendieck, \emph{Sur quelques points d'alg\`ebre homologique},
  Tohoku Math. J. (2) \textbf{9} (1957), 119--221.

\bibitem{HartshorneDef}
Robin Hartshorne, \emph{Deformation theory}, Graduate Texts in Mathematics,
  vol. 257, Springer, New York, 2010.

\bibitem{HT}
Brendan Hassett and Yuri Tschinkel, \emph{Weak approximation over function
  fields}, Invent. Math. \textbf{163} (2006), no.~1, 171--190.

\bibitem{HirzKod}
F.~Hirzebruch and K.~Kodaira, \emph{On the complex projective spaces}, J. Math.
  Pures Appl. (9) \textbf{36} (1957), 201--216.

\bibitem{Hirzebruch1954}
Friedrich Hirzebruch, \emph{Some problems on differentiable and complex
  manifolds}, Ann. of Math. (2) \textbf{60} (1954), 213--236.

\bibitem{HVF}
Nobuhiro Honda and Jeff Viaclovsky, \emph{Some remarks on fibrations in complex
  geometry}, J. Math. Study \textbf{58} (2025), no.~4, 552--574.

\bibitem{HV}
\bysame, \emph{Fibrations on the 6-sphere and {C}lemens threefolds}, Adv. Math.
  \textbf{503} (2026), Paper No. 111221, 37 pages.

\bibitem{HKP}
Alan~T. Huckleberry, Stefan Kebekus, and Thomas Peternell, \emph{Group actions
  on {$S^6$} and complex structures on {$\bold P_3$}}, Duke Math. J.
  \textbf{102} (2000), no.~1, 101--124. \MR{1741779}

\bibitem{KaupKaup}
Ludger Kaup and Burchard Kaup, \emph{Holomorphic functions of several
  variables: An introduction to the fundamental theory}, De Gruyter Studies in
  Mathematics, vol.~3, Walter de Gruyter \& Co., Berlin, 1983.

\bibitem{KodairaII}
Kunihiko Kodaira, \emph{On compact analytic surfaces. {II}}, Ann. of Math. (2)
  \textbf{77} (1963), no.~3, 563--626.

\bibitem{KodairaIII}
\bysame, \emph{On compact analytic surfaces. {III}}, Ann. of Math. (2)
  \textbf{78} (1963), no.~1, 1--40.

\bibitem{KodairaStability}
\bysame, \emph{On stability of compact submanifolds of complex manifolds},
  Amer. J. Math. \textbf{85} (1963), no.~1, 79--94.

\bibitem{Kollar}
J\'{a}nos Koll\'{a}r, \emph{Rational curves on algebraic varieties}, Ergebnisse
  der Mathematik und ihrer Grenzgebiete (3), vol.~32, Springer-Verlag, Berlin,
  1996.

\bibitem{LRS}
Christian Lehn, S\"{o}nke Rollenske, and Caren Schinko, \emph{The complex
  geometry of a hypothetical complex structure on {$S^6$}}, Differential Geom.
  Appl. \textbf{57} (2018), 121--137. \MR{3758366}

\bibitem{LR26}
Wenfei Liu and S{\"o}nke Rollenske, \emph{Yet another family of complex
  structures on {$S^6$}},
  \href{https://arxiv.org/abs/2609.33826}{arXiv:2609.33826}, 2026.

\bibitem{Mumford}
David Mumford, \emph{An analytic construction of degenerating abelian varieties
  over complete rings}, Compositio Math. \textbf{24} (1972), 239--272.

\bibitem{OS}
Keiji Oguiso and Tetsuji Shioda, \emph{The {M}ordell-{W}eil lattice of a
  rational elliptic surface}, Comment. Math. Univ. St. Paul. \textbf{40}
  (1991), no.~1, 83--99.

\bibitem{RYY}
Sheng Rao, Song Yang, and Xiangdong Yang, \emph{Dolbeault cohomologies of
  blowing up complex manifolds}, J. Math. Pures Appl. (9) \textbf{130} (2019),
  68--92.

\bibitem{Raynaud}
M.~Raynaud, \emph{Sp\'{e}cialisation du foncteur de {P}icard}, Inst. Hautes
  \'{E}tudes Sci. Publ. Math. (1970), no.~38, 27--76.

\bibitem{SerreGAGA}
Jean-Pierre Serre, \emph{G\'{e}om\'{e}trie alg\'{e}brique et g\'{e}om\'{e}trie
  analytique}, Ann. Inst. Fourier (Grenoble) \textbf{6} (1955/56), 1--42.

\bibitem{Tosatti}
Valentino Tosatti, \emph{Uniqueness of {$\Bbb{C}\Bbb{P}^n$}}, Expo. Math.
  \textbf{35} (2017), no.~1, 1--12.

\bibitem{Via}
Jeff Viaclovsky, \emph{A two-parameter family of complex structures on
  {$S^6$}}, \href{https://arxiv.org/abs/2609.33785}{arXiv:2609.33785}, 2026.

\bibitem{YauReview}
S.-T. Yau, \emph{Review of geometry and analysis}, Asian J. Math. \textbf{4}
  (2000), no.~1, 235--278, Kodaira's issue.

\bibitem{YauCalabi}
Shing~Tung Yau, \emph{Calabi's conjecture and some new results in algebraic
  geometry}, Proc. Nat. Acad. Sci. U.S.A. \textbf{74} (1977), no.~5,
  1798--1799.

\end{thebibliography}
\end{document}